\documentclass[10pt]{amsart}
\usepackage{color}
\usepackage{enumerate}
\usepackage{amssymb}
\usepackage{multirow}

\makeatletter 
\def\@tocline#1#2#3#4#5#6#7{\relax
  \ifnum #1>\c@tocdepth 
  \else
    \par \addpenalty\@secpenalty\addvspace{#2}%
    \begingroup \hyphenpenalty\@M
    \@ifempty{#4}{%
      \@tempdima\csname r@tocindent\number#1\endcsname\relax
    }{%
      \@tempdima#4\relax
    }%
    \parindent\z@ \leftskip#3\relax \advance\leftskip\@tempdima\relax
    \rightskip\@pnumwidth plus4em \parfillskip-\@pnumwidth
    #5\leavevmode\hskip-\@tempdima
      \ifcase #1
       \or\or \hskip 1em \or \hskip 2em \else \hskip 3em \fi%
      #6\nobreak\relax
    \dotfill\hbox to\@pnumwidth{\@tocpagenum{#7}}\par
    \nobreak
    \endgroup
  \fi}
\makeatother

\newtheorem{theorem}{Theorem}[section]
\newtheorem*{maintheorem}{Main Theorem}
\newtheorem{lemma}[theorem]{Lemma}
\newtheorem{proposition}[theorem]{Proposition}
\newtheorem{corollary}[theorem]{Corollary}

\theoremstyle{definition}
\newtheorem{definition}[theorem]{Definition}
\newtheorem{example}[theorem]{Example}
\newtheorem{remark}[theorem]{Remark}

\newtheorem{assumption}[theorem]{Assumption}

\newcommand{\R}{\mathbb R}

\newcommand{\Z}{\mathbb Z}
\newcommand{\C}{\mathbb C}

\newcommand{\So}{\mathcal{SR}}
\newcommand{\Sm}{\mathcal{SEM}}

\newcommand{\SF}{\mathbb S}

\DeclareMathOperator{\pv}{\wedge \mspace{-9.5mu}_* \ }

\newcommand{\N}{\mathbb N}

\newcommand{\HH}{\mathbb H}

\numberwithin{equation}{section}

\begin{document}

\title[Equivalence of slice semi-regular functions via Sylvester operators]{Equivalence of slice semi-regular functions via Sylvester operators}
\author[A. Altavilla]{A. Altavilla}\address{Altavilla Amedeo: Dipartimento di Matematica, Universit\`a degli Studi di Bari ``Aldo Moro'', via Edoardo Orabona, 4, 70125,
Bari, Italy.} \email{amedeo.altavilla@uniba.it}

\author[C. de Fabritiis]{C. de Fabritiis}\address{Chiara de Fabritiis: Dipartimento di Ingegneria Industriale e Scienze
Matematiche, Universit\`a Politecnica delle Marche, Via Brecce Bianche, 60131,
Ancona, Italia} \email{fabritiis@dipmat.univpm.it}
\thanks{The authors were partially supported by GNSAGA of INdAM; they acknowledge with pleasure the support of FBK-Cirm, Trento, where part of this paper was written.}


\subjclass[2010]{Primary 30G35; secondary 15A24, 15A33, 15A54, 39B42, 47A56}
\keywords{Slice-regular functions, Sylvester equation, matrix representation of Sylvester operators,
$*$-product of slice-regular functions, semi-regular functions, idempotent functions}

\begin{abstract} 
The aim of this paper is to study some features of slice semi-regular functions $\Sm(\Omega)$ on a circular domain $\Omega$ contained in the skew-symmetric algebra of quaternions $\HH$ via the analysis of a family of linear operators built from left and right $*$-multiplication on  $\Sm(\Omega)$; this class of operators includes the family of Sylvester-type operators $\mathcal{S}_{f,g}$.
Our goal is achieved by a strategy based on a matrix interpretation of these operators as we show that $\Sm(\Omega)$ can be seen as a $4$-dimensional vector space on the field $\Sm_\R(\Omega)$. 
We then study the rank of $\mathcal{S}_{f,g}$ and describe its kernel and image when it is not invertible, finding meaningful differences in the cases when the rank is either $2$ or $3$. By using these results, we are able to characterize when the functions $f$ and $g$ are either equivalent under $*$-conjugation or intertwined by means of a zero divisor, thus proving a number of statements on the behaviour of slice semi-regular functions. In this way, informations about the operator obtained by linear algebra techniques  give as a significant application the solution of a problem in an area of function theory which had an remarkable development in the last decade (see~\cite{G-S-St}). 
We also provide a complete classification of idempotents and zero divisors on product domains of $\HH$. 
\end{abstract}
\maketitle
\tableofcontents

\section{Introduction}

The aim of this article is to investigate the behaviour of slice semi-regular functions defined on a circular domain $\Omega$ contained in the skew-symmetric algebra of quaternions $\HH$ via the study of a family of Sylvester-type operators, and related equations; in particular, we single out such a family in a more general class of operators which are obtained as generalizations of left and right $*$-multiplication.
One of our main motives for this analysis is the fact that these operators are of crucial importance in the investigation of the  orbit of slice (semi)-regular functions under conjugation.
In such manner, the interplay between linear algebra and operator theory gives new and unexpected results under the function theoretical viewpoint.

In the most common use, Sylvester equations are special matrices equations, introduced by Sylvester himself~\cite{sylvester},
which are used in several subjects, including similarity, commutativity, control theory and differential equation (see~\cite{howandwhy}). In the quaternionic setting, such equations were studied with different purposes: 
without claiming any completeness of references, we point out the works of Bolotnikov~\cite{bolotnikov1,bolotnikov2} and Janovsk\'a--Opfer~\cite{janovska} regarding the quaternionic matricial equation and He--Liu--Tam~\cite{he} and references therein for the multitude of employments in applied sciences. For the operatorial equation in quaternionic function spaces we mention~\cite[Chapter 4]{ACSbook} and references therein.

In our paper,
we make a large use of a detailed analysis of the Sylvester operator 
in order to understand when two functions belong to the same conjugacy class under the action of an invertible element of $\Sm(\Omega)$. The deep interlacement between the function theory in $\Sm(\Omega)$ and the techniques of linear algebra used to study the behaviour of Sylvester operators answers several open questions concerning slice semi-regular functions; in particular it gives a necessary and sufficient conditions on a function in order it is conjugated to a one-slice preserving function (see Proposition~\ref{idemconjug}).

We now give an outline of the plan of the paper.
Section~\ref{preliminary} contains definitions and preliminary material: here we recall properties of slice semi-regular functions, the definition of the $*$-product and the interpretation
given in terms of the operators $\langle .,.\rangle_{*}$ and $\pv$ defined and developed in~\cite{A-dF,A-dFAMPA}.
Moreover, following the approach originally due to Colombo, Gonzales-Cervantes and Sabadini, we prove that the family $\Sm(\Omega)$ of slice semi-regular functions
on a symmetric domain is in fact a vector space over the field $\Sm_{\R}(\Omega)$ of slice semi-regular functions 
that preserves all the complex lines in $\HH$ (see Proposition~\ref{vectorspace}). Thanks to this result we can write
any slice semi-regular function $f$ as a sum $f=f_{0}+f_{v}$, where $f_{0}\in\Sm_{\R}(\Omega)$ can be interpreted as the ``real part'' of $f$ and $f_{v}$ as the ``vector part'' of $f$.
Afterwards we deal with idempotents
for the $*$-product: in particular we prove (see Proposition~\ref{characterizationidempotent}), that any semi-regular idempotent $f\in\Sm(\Omega)$ is regular and that $f$ is an idempotent if and only if it is a zero divisor whose ``real part" $f_0$ is identically equal to $\frac12$. This characterization allows us to describe all zero divisors in terms of idempotents in Propositions~\ref{propzerodividem}.

In the next section 
we define the class of $\Sm_\R(\Omega)$-linear operators $\mathcal{L}_{\mathcal{F},\mathcal{G}}:\Sm(\Omega)\to\Sm(\Omega)$ as 
$$\mathcal{L}_{\mathcal{F},\mathcal{G}}(\chi):=f_{[1]}*\chi*g_{[1]}+\cdots f_{[N]}*\chi*g_{[N]},$$
for any $N$-tuples $\mathcal{F}=(f_{[1]}\dots,f_{[N]}),\mathcal{G}=(g_{[1]},\dots,g_{[N]})\subset(\Sm(\Omega)\setminus\{0\})^N$.
We then study the initial case $N=1$, that is the multiplicative operators given by $\mathcal{L}_{f,g}(\chi)=f*\chi*g$; in particular we characterize when $\mathcal{L}_{f,g}$ is an isomorphism
(see Proposition~\ref{isomorphism});  in Theorem~\ref{singularcase} and Proposition~\ref{kernelfg} we describe the image and the kernel of this operator when it is not an isomorphism.

In Section~\ref{representation} we present a matrix interpretation of the linear operator $\mathcal{L}_{\mathcal{F},\mathcal{G}}$ via coordinates, being thus able to find necessary and sufficient conditions on $\mathcal{F},\mathcal{G}$ in order that $\mathcal{L}_{\mathcal{F},\mathcal{G}}$ is an isomorphism.
We later turn to the study of the Sylvester operators, which correspond to the case $\mathcal{F}=(f,1)$ and $\mathcal{G}=(1,g)$, thus giving $\mathcal{S}_{f,g}(\chi)=f*\chi+\chi*g$. 
After defining the equivalence relation $\simeq$ given by 
$f\simeq g$ when there exists an invertible $h$ such that $g=h^{-*}*f*h$, we prove that $\mathcal{S}_{f,g}$ is not an isomorphism if and only if either $f\simeq -g$ or there exist a zero divisor $\sigma$ such that $f*\sigma+\sigma*g=0$.

Section~\ref{ranksylv} contains a detailed analysis of the rank of the Sylvester operator according to the features of $f$ and $g$. 
We prove that the rank of $\mathcal{S}_{f,g}$ is always strictly greater than $1$ and show that it is not an isomorphism if and only if 
$$
(f_0+g_0)^2[(f_0+g_0)^2+2(f_v^s+g_v^s)]+(f_v^s-g_v^s)^2\equiv 0,
$$
where $f_{v}^{s}$ and $g_{v}^{s}$ denote the symmetrized functions of $f_{v}$ and $g_{v}$.
In particular we prove that
\begin{proposition}
If $\Omega$ is a slice domain, then the following conditions are equivalent
\begin{itemize}
\item $f\simeq g$, 
\item $f_0=g_0$ and $f^s=g^s$,
\item $\mathcal{S}_{f,-g}$ is not an isomorphism.
\end{itemize}
\end{proposition}
We then show (see Proposition~\ref{ranksimeq} and Theorem~\ref{rank3}) the following characterization  of the rank of $\mathcal{S}_{f,g}$  in terms of the ``real parts'' of the functions $f$ and $g$.
\begin{proposition}
Suppose that $\mathcal{S}_{f,g}$ is not an isomorphism. If $f_0+g_0\equiv0$ the operator $\mathcal{S}_{f,g}$ has rank $2$, otherwise it has rank $3$.
\end{proposition}

The succeeding section is devoted to the study of the Sylvester operators of maximal rank. In this case we are able to write explicitly   the solution of the equation $S_{f,g}(\chi)=\mathfrak{b}$ in terms of suitable functions $\lambda_L$ and $\lambda_R$ built by means of $f$ and $g$.
Section~\ref{rank2section} contains the final characterization of the equivalence relation $\simeq$: after describing the kernel of $\mathcal{S}_{f,g}$ when $f_0=-g_0$ and $f_v^s=g_v^s$, we show (see Theorem~\ref{kerthm}) that it contains invertible elements. This proves that the relation $f\simeq g$ holds if and only if $f_0=g_0$ and $f_v^s=g_v^s$, even when $\Omega$ is a product domain. We are also able to find conditions on $f$ and $g$ such that the kernel of the operator $\mathcal{S}_{f,g}$ contains zero divisors and to give a detailed picture of the image of $\mathcal{S}_{f,g}$.

Thanks to the results obtained on Sylvester operators of rank $2$, in Section~\ref{outcome} we characterize when a slice semi-regular function is equivalent to a one-slice preserving function, namely this happens if and only if $f_v^s$ has a square root. In particular this implies that all idempotents are equivalent. Last result allows us to give a different and more detailed description of the kernel of $\mathcal{L}_{f,g}$ when both $f$ and $g$ are idempotents.
Finally, Section~\ref{rank3sect} contains a detailed description of the couples of functions $f,g$ such that $\mathcal{S}_{f,g}$ has rank~$3$. 

In order to give a coincise overview of the relation which holds between the features of the couple $(f,g)$ and the behaviour of the Sylvester operator $\mathcal{S}_{f,g}$, we summarize the results
 of Sections $5-9$ in the following statement:
\begin{maintheorem}
Let $f,g\in\Sm(\Omega)\setminus\Sm_{\R}(\Omega)$. Then $\text{rk}(\mathcal{S}_{f,g})$ is always strictly greater than $1$.
Moreover we have
\begin{itemize}
\item $\text{rk}(\mathcal{S}_{f,g})=4$ $\Leftrightarrow$ $(f_{0}+g_{0})^{2}[(f_{0}+g_{0})^{2}+2(f_{v}^{s}+g_{v}^{s})]+(f_{v}^{s}-g_{v}^{s})^{2}\not\equiv 0$;
\item  $\text{rk}(\mathcal{S}_{f,g})=3$ $\Leftrightarrow$ $(f_{0}+g_{0})^{2}[(f_{0}+g_{0})^{2}+2(f_{v}^{s}+g_{v}^{s})]+(f_{v}^{s}-g_{v}^{s})^{2}\equiv 0$ and $f_{0}+g_{0}\not\equiv 0$ $\Leftrightarrow$ $\ker(\mathcal{S}_{f,g})$ contains only zero divisors (this case can occur only if $\Omega$ is a product domain);
\item $\text{rk}(\mathcal{S}_{f,g})=2$ $\Leftrightarrow$ $f\simeq -g$ $\Leftrightarrow$ $f_{0}+g_{0}\equiv0$ and $f_{v}^{s}\equiv g_{v}^{s}$
$\Leftrightarrow$ $\ker(\mathcal{S}_{f,g})$ contains at least an invertible element in $\Sm(\Omega)$. 
\end{itemize}
In last case, $\ker(\mathcal{S}_{f,g})$ contains also zero divisors if and only if $\Omega$ is a product domain
and one of the following holds
\begin{enumerate}
\item $f_{v}=g_{v}$ and $f_{v}^{s}$ has a square root;
\item $f_{v}\neq g_{v}$ and $(f_{v}-g_{v})^{s}\equiv 0$;
\item  $(f_{v}-g_{v})^{s}\not\equiv 0$ and $f_{v}^{s}$ has a square root.
\end{enumerate}
\end{maintheorem}

The authors warmly thank the anonymous referee for its accurate suggestions which improved the quality of the paper.

\section{Preliminary results}\label{preliminary}

In this section we recall some basic notion and result on slice regular and semi-regular functions and prove a couple of preliminary results.
We start by recalling some relevant subset of $\HH$ and the family of domains where we will define our functions.
In the space of quaternions we denote by $i,j,k$ the usual defining basis, so that any quaternion $q\in\HH$
can be written as $q=q_{0}+q_{1}i+q_{2}j+q_{3}k$, where $q_{\ell}\in\R$, $\ell=0,1,2,3$, and $i,j,k$ satisfy $i^{2}=j^{2}=k^{2}=-1$ and $ij=-ji=k$. If $q=q_{0}+q_{1}i+q_{2}j+q_{3}k$, then its usual quaternionic conjugate will be denoted by $q^{c}=q_{0}-(q_{1}i+q_{2}j+q_{3}k)$. The square norm of $q$ is then given by $|q|^{2}=qq^{c}$.
The set of imaginary units, i.e. the set of quaternions whose square equals $-1$, is denoted by $\SF$:
$$
\SF:=\{q\in\HH\,|\,q^{2}=-1\}.
$$
For any $q=q_{0}+q_{1}i+q_{2}j+q_{3}k\in\HH$, we define its vector part as $q_{v}=(q-q^{c})/2$, hence $q=q_{0}+q_{v}$.
Moreover, if $q_{v}\neq0$, we can also write $q=q_{0}+|q_{v}|\frac{q_{v}}{|q_{v}|}$ and $\left(\frac{q_{v}}{|q_{v}|}\right)^{2}=-1$.
Thus, for any $q\in\HH$, we have $q=x+Iy$, where $I\in\SF$, $x=q_{0},y=|q_{v}|\in\R$.
It is then clear that the space of quaternions can be unfolded as $\HH=\cup_{I\in\SF}\C_{I}$, where
$$
\C_{I}:=\mbox{Span}_{\R}(1,I)=\{x+Iy\,|\,x,y\in\R\}.
$$
Given $q=x+Iy\in\HH$, we set $\SF_{q}:=\{x+Jy\,|\,J\in\SF\}$. 
\begin{definition}
We say that a domain $\Omega\subset \HH$ is
\textit{circular}, if, for any $q=x+Iy\in\Omega$, we have that $\SF_{q}\subset\Omega$. If $\Omega \cap\R\neq \emptyset$,
a circular domain $\Omega$ is called a \textit{slice domain}, otherwise it is called
a \textit{product domain}.
\end{definition}
For any circular set $\Omega\subset \HH$ and $I\in\SF$, we write $\Omega_{I}=\Omega\cap\C_{I}$ and $\Omega_{I}^{+}=\Omega\cap\C_{I}^{+}$,
where $\C_{I}^{+}:=\{x+Iy\,|\,x\in\R,y>0\}$.
A subset of $\Omega$ of the form $\Omega_{I}$ (respectively $\Omega_{I}^{+}$) will be called a \textit{slice} (respectively a \textit{semi-slice}) of $\Omega$.
Notice that, if $\Omega$ is a product domain, then, for any $I\in\SF$, we have $\Omega=\Omega_{I}^{+}\times \SF$.

We have now set up all the notation we need to recall the definition of regularity (for an extensive approach to the subject of slice regular functions see~\cite{C-S-St-1,C-S-St-2,G-S-St}).
\begin{definition}
Let $\Omega\subset\HH$ be a circular domain. A function $f:\Omega\to\HH$ is said to be \textit{slice regular} if all its restrictions
$f_{I}=f|_{\Omega_{I}}$ are real differentiable and holomorphic, i.e., for any $I\in\SF$, it holds
$$
\frac{1}{2}\left(\frac{\partial}{\partial x}+I\frac{\partial}{\partial y}\right)f_{I}(x+Iy)\equiv 0.
$$
The family of slice regular functions over a fixed circular domain $\Omega$ will be denoted by $\So(\Omega)$.
\end{definition}
Examples of slice regular functions are given by polynomials with quaternionic coefficients on the right. 
Among the several properties that have been proved for slice regular functions we will make use of the so-called
\textit{Identity Principle}~\cite{A-CVEE, G-S-St, stoppatopoles} stating that if a regular function $f$ equals $0$ on a set containing two accumulation points living in two different semi-slices then $f\equiv 0$. In particular, if $f$ vanishes identically on an open set, then
it vanishes everywhere.

It is well known that pointwise product does not preserve slice regularity. This issue can be solved
by working with the $*$-product which generalizes the usual product of polynomials on a ring. Given
$f,g\in\So(\Omega)$, we define $f*g\in\So(\Omega)$ as
$$
(f*g)(q)=\begin{cases}
0,&\mbox{if } f(q)=0,\\
f(q)g(f(q)^{-1}qf(q)),&\mbox{otherwise}.
\end{cases}
$$
In general, the $*$-product is not commutative, however, if $f$ and $g$ are such that there exists $I\in\SF$ for which
$f(\Omega_{I})\subset\C_{I}$ and $g(\Omega_{I})\subset \C_{I}$, then $f*g=g*f$. Moreover, if $f$ is such that 
for any $I\in\SF$ $f(\Omega_{I})\subset C_{I}$, then $f*g=g*f=fg$, for any $g\in\So(\Omega)$.
The previous properties characterize two remarkable sets of slice regular functions. 
\begin{definition}
A function $f\in\So(\Omega)$, such that
there exists $I\in\SF$ for which $f(\Omega_{I})\subset \C_{I}$ is said to be \textit{one slice preserving} or $\C_{I}$-\textit{preserving};
the set of $\C_{I}$-preserving regular functions is denoted by $\So_{I}(\Omega)$. A function $f\in\So(\Omega)$ such that
$f(\Omega_{I})\subset \C_{I}$, for any $I\in\SF$,  is said to be \textit{slice preserving}; the set of slice preserving regular functions is denoted by $\So_{\R}(\Omega)$. 
\end{definition}

A special regular function that will be widely used next is presented in the following definition.
\begin{definition}\label{mapJ}
 We define the slice regular function $\mathcal{J}:\HH\setminus\R\to\SF$ as $\mathcal{J}(q)=\frac{q_{v}}{|q_{v}|}$,
 for all $q\in\HH\setminus\R$.
\end{definition}

It is easily seen that $\mathcal{J}$ is slice preserving and slice constant in the sense of~\cite[Definition 13]{A-CVEE}.
Moreover, notice that $\mathcal{J}^{*2}=\mathcal{J}^{2}=-1$.

\begin{remark}
The function $\mathcal{J}$ given in Definition~\ref{mapJ} can be interpreted in the sense of stem functions (see~\cite{G-P}) as follows: let us consider the stem function $\mathsf{J}:\C\setminus\R\to\HH_{\C}$ 
$$
\mathsf{J}(z):=\begin{cases}
\imath,\qquad\mbox{if }z\in\C^{+}\\
-\imath,\quad\,\,\mbox{if }z\in\C^{-};
\end{cases}
$$
then $\mathsf{J}$ induces the slice regular function $\mathcal{J}=\mathcal{I}(\mathsf{J})$.
\end{remark}

\subsection{$\So$ as a 4-rank free module over $\So_{\R}$}

Complete $1$ to a basis $(1,I,J,K)$ of $\HH$.
We recall from \cite[Proposition 3.12]{C-GC-S} and~\cite[Lemma 6.11]{G-M-P}, that any slice regular function
$f\in\So(\Omega)$ can be written in a unique way as a sum $f=f_{0}+f_{1}I+f_{2}J+f_{3}K$, where $f_{\ell}\in\So_{\R}(\Omega)$, and $\ell=0,1,2,3$. In particular $\So(\Omega)$ is a 4-rank free module on $\So_{\R}(\Omega)$. 
Given $f\in\So(\Omega)$, by means of the previous formalism, it is possible to write the \textit{regular conjugate} $f^{c}$ and the \textit{symmetrized function} $f^{s}$ (see~\cite[Definition 1.33]{G-S-St}), as
\begin{equation}\label{regconj}
f^{c}=f_{0}+f_{1}I^{c}+f_{2}J^{c}+f_{3}K^{c},\qquad f^{s}=f*f^{c}.
\end{equation}
We assume henceforth $(1,I,J,K)$ to be an orthonormal basis. The previous formulas simplify as explained in~\cite[Remark 2.2]{A-dF} as
$$
f^{c}=f_{0}-(f_{1}I+f_{2}J+f_{3}K),\qquad f^{s}=f_{0}^{2}+f_{1}^{2}+f_{2}^{2}+f_{3}^{2}.
$$
A further  consequence of this result is a more intuitive
representation of the $*$-product, similar to the usual quaternionic product in its ``scalar-vector'' form. First of all, given $f\in\So(\Omega)$, notice that $f_{0}=(f+f^{c})/2$ and $f_{v}=f-f_{0}$ (in particular $f_{0}\equiv 0$ holds if and only if $f\equiv-f^{c}$). 
For any regular function $f$ we will sometimes informally call $f_{0}$ as its ``real part'' and $f_{v}$ as its ``vector part'',
even if $f_{0}$ and $f_{v}$ are quaternionic valued and not real or pure-imaginary valued functions. 
If $g=g_{0}+g_{v}$ is another element of $\So(\Omega)$, we have~\cite[Proposition 2.7]{A-dF}
\begin{equation}\label{formulaproduct}
f*g=f_{0}g_{0}-\langle f_{v},g_{v}\rangle_{*}+f_{0}g_{v}+g_{0}f_{v}+f_{v}\pv g_{v},
\end{equation}
where $\langle .,.\rangle_{*}$ and $\pv$ are defined as follows
\begin{equation}\label{scalarvector}
\langle f,g\rangle_{*}(q)=(f*g^{c})_{0}(q),\qquad (f\pv g)(q)=(f_{v}\pv g_{v})(q)=\frac{(f*g)(q)-(g*f)(q)}{2}.
\end{equation}

The following remark can be interpreted as a non degeneracy result of the 
``scalar product'' $\langle \cdot,\cdot\rangle_{*}$ given in formula~\eqref{scalarvector}.

\begin{remark}\label{rkdelta}
Notice that $(f\delta)_{0}\equiv0$ for all $\delta\in\HH$ with $|\delta|=1$ if and only if $f\equiv 0$. Indeed
if we choose an orthonormal basis $\{1,I,J,K\}$ of $\HH$ and write $f=f_{0}+f_{1}I+f_{2}J+f_{3}K$, 
we have 
$$ 
(f\cdot 1)_{0}\equiv f_{0},\quad (f\cdot i)_{0}\equiv -f_{1},\quad (f\cdot j)_{0}\equiv -f_{2},\quad (f\cdot k)_{0}\equiv -f_{3},
$$
and hence $f\equiv 0$.
\end{remark}

The representation of the $*$-product give in formula~\eqref{formulaproduct} makes possible  to prove the following result which will be useful in some of the computations to come.

%
\begin{lemma}
Let $f$ and $g$ be regular functions defined on the same domain $\Omega$. Then we have 
$$
(f+g)^{s}=f^{s}+g^{s}+2\langle f,g\rangle_{*}.
$$
\end{lemma}

\begin{proof}
The following chain of equalities yields the thesis
\begin{align*}
(f+g)^{s}=(f+g)*(f^{c}+g^{c})&=f*f^{c}+f*g^{c}+g*f^{c}+g*g^{c}\\
&=f^{s}+g^{s}+f*g^{c}+(f*g^{c})^{c} \\
&=f^{s}+g^{s}+2(f*g^{c})_{0}=f^{s}+g^{s}+2\langle f,g\rangle_{*}.
\end{align*}
\end{proof}

\subsection{Semi-regular functions}

Another interesting property of a regular function $f$ is the structure of its zero set $V(f)$~\cite{G-S,G-S-St, G-P,GPSalgebra} and 
of its singularities~\cite{G-S-St, GPSadvances,GPSdivision,stoppatopoles,stoppatosing}. 
 Ghiloni, Perotti and Stoppato proved the following statement in~\cite[Theorem 3.5]{GPSdivision}, generalizing results due to several authors.
\begin{theorem}[Ghiloni-Perotti-Stoppato]
Assume that $\Omega$ is either a slice or a product domain and let $f\in\So(\Omega)$.
\begin{itemize}
\item  If $f\not\equiv 0$ then the intersection $V(f)\cap\C_{J}^{+}$ is closed and discrete in $\Omega_{J}$ for all
$J\in\SF$ with at most one exception $J_{0}$, for which it holds $f|_{\Omega_{J_{0}}^{+}}\equiv 0$.
\item  If $f^{s}\not\equiv 0$ then the set $V(f)$ is a union of isolated points or isolated spheres of the form $\SF_{q}$.
\item If $\Omega$ is a slice domain, then $f\not\equiv 0$ implies $f^{s}\not\equiv 0$.
\end{itemize}
\end{theorem}

%
%
%
%
%

In the same paper, Ghiloni, Perotti and Stoppato also developed a theory of singualarities for slice regular functions,
which is a consequence of a detailed study of Laurent expansions near spheres $\SF_{q}$ and real points;
the notion of meromorphic function can thus be translated in this context as that of \textit{semi-regular function}. 
We now briefly recall the notions of removable singularity and pole at non real points; the case of real points
is completely analogous. For more detailed statements and complete proofs see~\cite[Section 6]{GPSdivision}.


 Let $\Omega$ be a circular domain and $p\in\Omega\setminus\R$.
Any $f\in\So(\Omega\setminus\SF_{p})$ can be written near $\SF_{p}$
as 
$$
f(q)=\sum_{n\in\Z}(q-p)^{*n}b_{n},\qquad f(q)=\sum_{\nu\in\Z}\Delta_{p}^{\nu}(q)(qu_{\nu}+v_{\nu}),
$$
with $b_{n},u_{\nu},v_{\nu}\in\HH$, for any $n$ and $\nu$. 
The point $p$ is said to be a \textit{pole} for $f$ if there exists an $n_{0}\geq 0$ such that $b_{n}=0$ for all $n<-n_{0}$, in particular if $f$ extends to a slice regular function in a circular open set containing $\SF_{p}$, $p$ called a removable singularity; the minimum of the above $n$ is called the \textit{order of the pole} and denoted as $\text{ord}_{f}(p)$. 
If $p$ is neither a removable singularity nor a pole, then it is called an \textit{essential singularity} for $f$ and $\text{ord}_{f}(p)$ is set to be $+\infty$. Finally, the \textit{spherical order} of $f$ at $\SF_{p}$ is the smallest even natural
number $2\nu_{0}$ such that $u_{\nu}=v_{\nu}=0$ for all $\nu<-\nu_{0}$. If no such $\nu_{0}$ exists, then we set $\text{ord}_{f}(\SF_{p})=+\infty$.

Non-real singularities for slice regular functions can be classified as follows (see~\cite[Theorem 6.4]{GPSdivision}). Let $\Omega$ be
a circular domain, $p\in\Omega\setminus\R$ and set $\widetilde\Omega:=\Omega\setminus \SF_{p}$. If 
$f\in\So(\widetilde\Omega)$ then one of the following holds:
\begin{itemize}
\item every point of $\SF_{p}$ is a removable singularity for $f$; in this case $\text{ord}_{f}(\SF_{p})=0=\text{ord}_{f}(w)$, for any $w\in\SF_{p}$;
\item every point of $\SF_{p}$ is a non removable pole for $f$. There exists $n\in\N\setminus\{0\}$ such that the function $\Delta_{p}^{n}(q)f(q)$ extends to a slice regular function $g$ defined on $\Omega$ that has at most one zero in $\SF_{p}$; in this case
$\text{ord}_{f}(\SF_{p})=2k$; 
moreover, $\text{ord}_{f}(w)=k$ and $\lim_{\Omega\ni x\to w}|f(x)|=+\infty$ for all $w\in\SF_{p}$ except the possible zero of $g$, at which $\text{ord}_{f}$ must be less than $k$;
\item every point of $\SF_{p}$, except at most one, is an essential singularity for $f$; in this case
$\text{ord}_{f}(\SF_{p})=+\infty$ and there exists at most one point $w\in\SF_{p}$ such that $\text{ord}_{f}(w)<\infty$.
\end{itemize}
In the special case of a slice preserving function $f$, for any point $\tilde p$ belonging to
the sphere $\SF_{p}$, it holds $\text{ord}_{f}(\SF_{p})=2\text{ord}_{f}(\tilde p)$, i.e. all the points of $\SF_{p}$ have the same order.

Notice that, the set of singularities has different structure with respect to the zero set: indeed there are no non-real isolated singular points for a slice regular function. We now give the definition of semi-regular function.

\begin{definition}
A function $f$ is said to be \textit{slice semi-regular} in a nonempty circular domain $\Omega$, if there exists
a circular open subset $\widetilde\Omega\subseteq\Omega$ such that $f\in\So(\widetilde \Omega)$ and such that
each point of $\Omega\setminus\widetilde\Omega$ is a pole or a removable singularity for $f$.
The set of slice semi-regular functions on $\Omega$ will be denoted as $\Sm(\Omega)$; the sets of slice preserving
and of $\C_{I}$-preserving (for some $I\in\SF$) semi-regular functions on $\Omega$ as $\Sm_{\R}(\Omega)$ and $\Sm_{I}(\Omega)$, respectively.
\end{definition}

%

\subsection{$\Sm$ as a 4-dimensional vector space over $\Sm_{\R}$}

We now pass to analyze some algebraic properties of $\Sm(\Omega)$. First of all consider the action 
 $\So_{\R}(\Omega)\times\So(\Omega)\to\So(\Omega)$, given by $(f,g)\mapsto f*g=fg$. 
Thanks to the Identity Principle and the fact that the zero set of a non-constant regular function has empty interior,
the equality $fg\equiv 0$ implies that either $f$ or $g$ is identically zero (this is a special case of \cite[Proposition 3.8]{GPSdivision}). In particular $(\So_{\R}(\Omega),+,*)$ is an integral domain and $\Sm_{\R}(\Omega)$ is a field. 
Moreover, recalling~\cite[Theorem 6.6]{GPSdivision}, we have that  if $\Omega$ is a slice domain then $\Sm(\Omega)$ is a division algebra and, also when $\Omega$ is a product domain, any $f\in\Sm(\Omega)$ such that $f^{s}\not\equiv 0$ has a multiplicative inverse given by $f^{-*}=(f^{s})^{-1}f^{c}$.

In the case of semi-regular functions, we can describe the structure of the algebra $\Sm(\Omega)$ adjusting to this situation the 
already mentioned results given in~\cite[Proposition 3.12]{C-GC-S} and~\cite[Lemma 6.11]{G-M-P}.  

\begin{proposition}\label{vectorspace}
Let $(1,I,J,K)$ be a basis of $\HH$. The map
$$
(f_{0},f_{1},f_{2},f_{3})\ni(\Sm_{\R}(\Omega))^{4}\mapsto f_{0}+f_{1}I+f_{2}J+f_{3}K\in\Sm(\Omega)
$$
is a $\Sm_{\R}(\Omega)$-linear isomorphism.
In particular $\Sm(\Omega)$ is a $4$-dimensional vector space on $\Sm_{\R}(\Omega)$.
\end{proposition}

\begin{proof}
Let $f\in\Sm(\Omega)$. Let $\Omega'$ be
a circular subdomain of $\Omega$ such that $f\in\So(\Omega')$ and such that every point of $\Omega\setminus \Omega'$
is a pole for $f$. Proposition 3.12 in~\cite{C-GC-S} guarantees the existence of a unique 4-tuple
$f_{0},f_{1},f_{2},f_{3}\in\So_{\R}(\Omega')$ such that $f=f_{0}+f_{1}I+f_{2}J+f_{3}K$. 
We are left with proving that $f_{0},\dots, f_{3}\in\Sm_{\R}(\Omega)$. If $\SF_{q_{0}}$ is a spherical pole of $f$ then there exists $m\in\N$
such that $\Delta_{q_{0}}^{m}\cdot f$ extends regularly in an open circular neighborhood $\mathcal{U}\subset\Omega$ of the sphere $\SF_{q_{0}}$. Now consider the function $\Delta_{q_{0}}^{m}\cdot f$ and apply again~\cite[Proposition 3.12]{C-GC-S}, finding
$g_{0},\dots, g_{3}\in\So_{\R}(\mathcal{U})$ such that $\Delta_{q_{0}}^{m}\cdot f=g_{0}+g_{1}I+g_{2}J+g_{3}K$.
Nonetheless we also have $\Delta_{q_{0}}^{m}\cdot f=\Delta_{q_{0}}^{m}\cdot f_{0}+\Delta_{q_{0}}^{m}\cdot f_{1}I+\Delta_{q_{0}}^{m}\cdot f_{2}J+\Delta_{q_{0}}^{m}\cdot f_{3}K$ on $\mathcal{U}\setminus\SF_{q_{0}}$ and the uniqueness given in~\cite[Proposition 3.12]{C-GC-S} 
ensures $\Delta_{q_{0}}^{m}\cdot f_{n}=g_{n}$ on $\mathcal{U}\setminus\SF_{q_{0}}$, for $n=0,1,2,3$.
Last equality shows that $f_{0},\dots f_{3}$ have a pole at $\SF_{q_{0}}$. The case of 
a real pole is treated analogously, showing that $f_{0},\dots f_{3}$ belong to $\Sm_{\R}(\Omega)$.
\end{proof}

The uniqueness of the above statement gives as an immediate consequence that $\Sm_{\R}(\Omega)$ is the center of $\Sm(\Omega)$ and that $\So_{\R}(\Omega)$ is the center of $\So(\Omega)$.

\begin{remark}
The above proof shows that if $f=f_{0}+f_{1}I+f_{2}J+f_{3}K\in\Sm(\Omega)$ has a sphere of poles $\SF_{q_{0}}$ of spherical order $2m$, then any point of $\SF_{q_{0}}$ is 
a pole of spherical order at most $2m$ or a removable singularity for each of the functions $f_{0},\dots, f_{3}$
and that $\text{ord}_{f}(\SF_{q_{0}})=\max\{\text{ord}_{f_{0}}(\SF_{q_{0}}),\dots, \text{ord}_{f_{3}}(\SF_{q_{0}})\}$.
\end{remark}

\subsection{Zero divisors and idempotents}

From~\cite[Theorem 6.6]{GPSdivision} we have that $\Sm(\Omega)$ contains zero divisors if and only if $\Omega$ is a product
domain (for a thorough study of the zero set of zero divisors see~\cite{G-P-wing}, while \cite[Example 3]{altavilladiff} contains explicit computations for relevant examples; for an interesting application of idempotents, \textit{i.e.} $f\not\equiv 0,1$ such that $f^{*2}=f$, to function spaces, see~ \cite{M-S} which
sets questions raised in~\cite{deF-G-S}).
In this case $f$ is a zero divisor if and only if $f^{s}\equiv 0$. 
In the sequel of this paper, we will often make use of the ``basic'' idempotents given in the following definition.
\begin{definition}
Let $\Omega$ be any product domain and $I\in\SF$. We define $\ell^{+,I}:\Omega\to\HH$ and $\ell^{-,I}:\Omega\to\HH$
as 
$$
\ell^{+,I}(x+Jy)=\frac{1-\mathcal{J}I}{2},\qquad\ell^{-,I}(x+Jy)=\frac{1+\mathcal{J}I}{2},
$$
where $y>0$.
\end{definition}
It is easily seen that $\ell^{+,I}$ and $\ell^{-,I}$ are idempotents and that the following equalities hold
(see~\cite[Remark2.4]{A-S}):
$$\left(\ell^{+,I}\right)^c=1-\ell^{+,I}=\ell^{-,I},\quad (\ell^{+,I})^{s}=(\ell^{-,I})^{s}=\ell^{+,I}*\ell^{-,I}\equiv 0$$

We now classify idempotents in $\Sm(\Omega)$ showing in particular that they have only removable singularities (and therefore, by a slight abuse of notation, we say they are regular).

\begin{proposition}\label{characterizationidempotent}
Let $f\in\Sm(\Omega)\setminus\{0,1\}$. The function $f$ is an idempotent for the $*$-product if and only if
$f$ belongs to $\So(\Omega)$ and it is a zero divisor such that $f_{0}\equiv \frac{1}{2}$ (and thus $f_{v}^{s}\equiv -\frac{1}{4}$).
\end{proposition}
\begin{proof}
Suppose $f\in\Sm(\Omega)$ is an idempotent. This can be written as $f^{*2}=f$. The previous equality can be written as $(f-1)*f\equiv 0$ 
which entails that $f$ is a zero divisor (since $f\not\equiv 0,1$). Using the splitting $f=f_{0}+f_{v}$
and the fact that $f_{v}*f_{v}=-f_{v}^{s}$ the equality $f^{*2}=f$ is equivalent to the system
\begin{equation}\label{idempotentsystem}
\begin{cases}
f_{0}^{2}-f_{v}^{s}=f_{0}\\
2f_{0}f_{v}=f_{v}.
\end{cases}
\end{equation}
Last equality can be also written as $(2f_{0}-1)f_{v}\equiv 0$ which gives either $f_{v}\equiv 0$ or $f_{0}\equiv \frac{1}{2}$.
The first case cannot hold since $\Sm_{\R}(\Omega)$ does not contain zero divisors; thus $f_{0}\equiv \frac{1}{2}$ and the first equality of system~\eqref{idempotentsystem} becomes $f_{v}^{s}\equiv -\frac{1}{4}$.
Then we are left with proving that $f$ is regular. Since $f_{0}\equiv\frac{1}{2}$ whenever defined, it can be extended regularly to 
the function $\frac{1}{2}$ on the domain $\Omega$, so it only has removable singularities. Now suppose $f_{v}$
has a spherical pole in $\SF_{q_{0}}$ of order $k$, thus there exists a function $g_{v}$ regular on a neighborhood $\mathcal{U}$ of $\SF_{q_{0}}$
which has at most one possible isolated zero in $\SF_{q_{0}}$ of order $\tilde k<k$, 
such that 
\begin{equation}\label{eqgv1}
g_{v}=\Delta_{q_{0}}^{k}\cdot f_{v},
\end{equation}
on $\mathcal{U}\setminus\SF_{q_{0}}$ (see~\cite[Theorem 6.4 (2)]{GPSdivision}). Thanks to~\cite[Theorem 22 and Remark 14]{G-P}  we can also write 
\begin{equation}\label{eqgv2}
g_{v}=(q-w_{1})*\dots*(q-w_{\tilde k})*\gamma,
\end{equation}
where $w_{1},\dots, w_{\tilde k}\in\SF_{q_{0}}$, $w_{n+1}\neq w_{n}^{c}$ ($n=1,\dots, \tilde k-1$) and $\gamma$ is never vanishing on $\SF_{q_{0}}$. Computing the symmetrized function $g_{v}^{s}$ from equalities~\eqref{eqgv1} and~\eqref{eqgv2},
we obtain 
$$
\Delta_{q_{0}}^{2\tilde k}\gamma^{s}=g_{v}^{s}=\Delta_{q_{0}}^{2k}f_{v}^{s}=-\frac{1}{4}\Delta_{q_{0}}^{2k}.
$$
Since $\gamma^{s}$ is never vanishing on $\SF_{q_{0}}$, we then obtain $\tilde k=k$ which is a contradiction to the above 
inequality. 
The case of a real pole is treated analogously.
This shows that $f_{v}$ has no poles and thus $f$ belongs to $\So(\Omega)$.

Straightforward computations show that if $f\in\So(\Omega)$ is such that $f_{0}\equiv \frac{1}{2}$ and $f^{s}\equiv 0$ (that is $f_{v}^{s}\equiv -\frac{1}{4}$), then
$f$ is an idempotent.
\end{proof}

The above statement allows us to give an explicit characterization of zero divisors in $\Sm(\Omega)$.

\begin{proposition}\label{propzerodividem}
Let $f\in\Sm(\Omega)$ be a zero divisor.
For any $\delta \in\HH$ such that $|\delta|=1$ and $(f\delta)_{0}\not\equiv 0$, there exits $\sigma=\sigma(\delta)\in\So(\Omega)$ idempotent,
such that
\begin{equation}\label{zerodividem}
f=2(f\delta)_{0}\sigma\delta^{c}.
\end{equation}
In particular, if $f_{0}\not\equiv 0$, we can write $f=(2f_{0})\sigma$ for a suitable idempotent $\sigma$.
\end{proposition}

\begin{proof}
Assume first that $f_{0}\not\equiv 0$, then $f_{0}^{-*}=f_{0}^{-1}\in\Sm(\Omega)$. Thus, if $f=f_{0}+f_{v}$, we have that
$f=(2f_{0})\sigma$, where 
$$\sigma=\frac{1}{2}+(2f_{0})^{-1}f_{v}.$$
As $f^{s}=4f_{0}^{2}\sigma^{s}\equiv 0$, we also have that $\sigma^{s}\equiv 0$, proving that $\sigma$ is a zero divisor. Moreover,
$\sigma_{0}\equiv \frac{1}{2}$ and Proposition~\ref{characterizationidempotent} shows that $\sigma\in\So(\Omega)$ is an idempotent.
Now choose $\delta\in\HH$ with $|\delta|=1$ be such that $(f\delta)_{0}\not\equiv 0$;
such a $\delta$ always exists thanks to Remark~\ref{rkdelta}.
The fact that $(f\delta)^{s}\equiv f^{s}\equiv 0$
entails that $f\delta$ is a zero divisor and therefore we can apply the above reasoning obtaining 
$$f\delta=2(f\delta)_{0}\sigma,$$ 
for a suitable idempotent $\sigma$ and the thesis follows by multiplying both member of the last equality on the right  by $\delta^{c}$.
\end{proof}

\begin{remark}\label{relationidem}
We notice that the proof of the above proposition shows that formula~\eqref{zerodividem} can be written
as soon as $(f\delta)_{0}\not\equiv 0$. If $\delta$ and $\tilde\delta$ are unitary quaternions such that
$(f\delta)_{0}\not\equiv0$ and $(f\tilde\delta)_{0}\not\equiv0$, then we have 
$$
f=2(f\delta)_{0}\sigma\delta^{c}=2(f\tilde\delta)_{0}\tilde \sigma\tilde\delta^{c},
$$
for $\sigma$ and $\tilde \sigma$ suitable idempotents. Thus we can write 
\begin{equation*}\label{eqrelidem}
\tilde \sigma=\gamma \sigma\delta'=\sigma\gamma\delta',
\end{equation*}
where $\gamma=(f\tilde\delta)_{0}^{-1}(f\delta)_{0}\in\Sm_{\R}(\Omega)$ and $\delta'=\delta^{c}\tilde\delta$ is a unitary quaternion.
\end{remark}

\begin{remark}
Given $f\in\Sm(\Omega)$ a zero divisor and $\eta$ a unitary quaternion such that $(f\eta)_{0}\not\equiv 0$, from formula~\eqref{zerodividem}, we can also write
\begin{equation}\label{zerodividemright}
f=2(f\eta)_{0} \sigma\eta^{c}=2(f\eta)_{0} \eta^{c}\eta*\sigma*\eta^{c}=2(f\eta)_{0} \eta^{c}*\rho,
\end{equation}
where $\rho=\eta*\sigma*\eta^{c}$ is again an idempotent. 
\end{remark}

The proof of Proposition~\ref{propzerodividem} shows that if $f$ is a zero divisor with $f_{0}\not\equiv 0$, then
we can choose $\delta=1$ and therefore formula~\eqref{zerodividem} simplifies to $f=2f_{0}\sigma$.

\section{$\Sm_{\R}$-linear endomorphisms}\label{lineq}

The aim of this section is to study a class of $\Sm_{\R}(\Omega)$-linear operators in the space of slice semi-regular functions; they will be represented via suitable matrices in Section~\ref{representation}.
The class of linear operators we are interested in is described as follows.
\begin{definition}
Consider two $N$-tuples $\mathcal{F}:=(f_{[1]},\dots, f_{[N]})$ and $\mathcal{G}:=(g_{[1]},\dots,g_{[N]})\subset\Sm(\Omega)\setminus\{0\}$.
We denote by $\mathcal{L}_{\mathcal{F},\mathcal{G}}:\Sm(\Omega)\to\Sm(\Omega)$ the $\Sm_{\R}(\Omega)$-linear operator given by
\begin{equation}\label{elementary}
\mathcal{L}_{\mathcal{F},\mathcal{G}}(\chi):=f_{[1]}*\chi*g_{[1]}+\cdots f_{[N]}*\chi*g_{[N]}.
\end{equation}
\end{definition}
In particular the analysis of the image and the kernel of such operators will give complete information on the existence
and uniqueness of the solution of the equation
\begin{equation*}
f_{[1]}*\chi*g_{[1]}+\cdots f_{[N]}*\chi*g_{[N]}=\mathfrak{b},
\end{equation*}
for $\mathfrak{b}\in\Sm(\Omega)$.

Since $\Sm_{\R}(\Omega)$ is the center of $\Sm(\Omega)$, then $(\Sm_{\R}(\Omega)\setminus\{0\})^{N}$ acts on the $N$-tuples
$\mathcal{F}$ and $\mathcal{G}$ of semi-regular functions as follows: given $\alpha=(\alpha_{[1]},\dots,\alpha_{[N]})\in(\Sm_{\R}(\Omega)\setminus\{0\})^{N}$ we denote by $\alpha\blacklozenge\mathcal{F}=(\alpha_{[1]}f_{[1]},\dots,\alpha_{[N]}f_{[N]})$ and 
$\alpha\lozenge\mathcal{G}=(\alpha_{[1]}^{-1}g_{[1]},\dots,\alpha_{[N]}^{-1}g_{[N]})$. A straightforward computation
shows that $\mathcal{L}_{\mathcal{F},\mathcal{G}}=\mathcal{L}_{\alpha\blacklozenge\mathcal{F},\alpha\lozenge\mathcal{G}}$,
so that, when needed, we can suppose that $\mathcal{G}$ contains only regular functions without real and spherical zeroes.

We start our investigation from the easiest case $N=1$; to simplify notation we denote $\mathcal{L}_{\{f\},\{g\}}$ by $\mathcal{L}_{f,g}$.
Our first result classifies the functions $f$ and $g$ such that $\mathcal{L}_{f,g}$ is a real linear isomorphism
and gives explicitly the solution of $\mathcal{L}_{f,g}(\chi)=\mathfrak{b}$ in the case the operator is an isomorphism.

\begin{proposition}\label{isomorphism}
Let $f,g\in\Sm(\Omega)\setminus\{0\}$.
\begin{enumerate}
\item Provided $g\in\So(\Omega)$ has neither real nor spherical zeroes, then $\mathcal{L}_{f,g}$ maps $\So(\Omega)$ 
to $\So(\Omega)$ if and only if $f\in\So(\Omega)$.
\item The operator $\mathcal{L}_{f,g}$ is a real linear isomorphism if and only if  neither $f$ nor $g$ are zero divisors.
\item If $\mathcal{L}_{f,g}$ is an isomorphism, for any $\mathfrak{b}\in\Sm(\Omega)$ the equation $\mathcal{L}_{f,g}(\chi)=\mathfrak{b}$
has the unique solution $\chi=f^{-*}*\mathfrak{b}*g^{-*}$.
\item If $\mathcal{L}_{f,g}$ is an isomorphism, then
the solution of $\mathcal{L}_{f,g}(\chi)=\mathfrak{b}$ belongs to $\So(\Omega)$ for any $\mathfrak{b}\in\So(\Omega)$
if and only if $f$ and $g$ are never vanishing.
\end{enumerate}
\end{proposition}

\begin{proof}
\textit{(1)} If $f\in\So(\Omega)$, then trivially $\mathcal{L}_{f,g}(\So(\Omega))\subseteq\So(\Omega)$. 
Vice versa, if $\mathcal{L}_{f,g}(\So(\Omega))\subseteq\So(\Omega)$, in particular we have that $\mathcal{L}_{f,g}(1)=f*g\in\So(\Omega)$. Since $g$ has neither real nor spherical zeroes, then $f$ has neither real nor spherical poles and therefore
$f\in\So(\Omega)$, too.

\textit{(2)} If $f$ is a zero divisor, then there exists $\chi_{f}\not\equiv 0$ such that $f*\chi_{f}\equiv 0$ and trivially
$\mathcal{L}_{f,g}(\chi_{f})=0$ so that $\mathcal{L}_{f,g}$ is not an isomorphism; the same holds for $g$. Vice versa, assume that $\mathcal{L}_{f,g}$
is not an isomorphism; then there exists $\chi\in\Sm(\Omega)\setminus\{0\}$ such that $\mathcal{L}_{f,g}(\chi)=f*\chi*g=0$.
If $f*\chi=0$, then $f$ is a zero divisor; otherwise the equality
$(f*\chi)*g=0$ gives that $g$ is a zero divisor.

\textit{(3)} Since $\mathcal{L}_{f,g}$ is an isomorphism, then $f$ and $g$ are not zero divisors and $f^{-*}$ and $g^{-*}$ belong to $\Sm(\Omega)$. A direct computation
shows that $\mathcal{L}_{f,g}(f^{-*}*\mathfrak{b}*g^{-*})=\mathfrak{b}$.

\textit{(4)} If $f,g\in\So(\Omega)$ are never vanishing, then \textit{(3)} shows that the unique solution of $\mathcal{L}_{f,g}(\chi)=\mathfrak{b}$
belongs to $\So(\Omega)$ for any $\mathfrak{b}\in\So(\Omega)$. Vice versa, if $f^{-*}*\mathfrak{b}*g^{-*}$ belongs to $\So(\Omega)$ 
for any $\mathfrak{b}\in\So(\Omega)$, by taking $\mathfrak{b}=g$ we obtain that $f^{-*}\in\So(\Omega)$, implying that 
$f$ has no zeroes; the same holds for $g$.
\end{proof}

Notice that if $\Omega$ is a slice domain, then
$\mathcal{L}_{f,g}$ is always an isomorphism thanks to \textit{(2)} of the above proposition.

In the case $\mathcal{L}_{f,g}$ is not an isomorphism we give a necessary and sufficient condition on the function $\mathfrak{b}$ 
in order it belongs to the image of $\mathcal{L}_{f,g}$.

\begin{theorem}\label{singularcase}
Let $f,g\in\Sm(\Omega)\setminus\{0\}$ be such that $\mathcal{L}_{f,g}$ is not an isomorphism.
If $f$ is a zero divisor, for a suitable unitary $\delta\in\HH$, we denote by $\sigma_f$ the idempotent given in formula~\eqref{zerodividem}.
If $g$ is a zero divisor, for a suitable unitary $\eta\in\HH$, we denote by $\rho_g$ the idempotent given in formula~\eqref{zerodividemright}. Then there exists $\chi$ such that $\mathcal{L}_{f,g}(\chi)=\mathfrak{b}$ if and only if $\mathfrak{b}=\sigma_f*\mathfrak{b}$, if 
$f$ is a zero divisor, and $\mathfrak{b}=\mathfrak{b}*\rho_g$, if $g$ is a zero divisor.
\end{theorem}

%
%

\begin{remark}
The relation $\mathfrak{b}=\sigma_f*\mathfrak{b}$ can also be written as $(1-\sigma_f)*\mathfrak{b}\equiv 0$ that is $\sigma_f^{c}*\mathfrak{b}\equiv 0$. Moreover,
thanks to Remark~\ref{relationidem}, this condition does not depend on the unitary quaternion $\delta$ appearing in 
formula~\eqref{zerodividem}. Indeed, if $\tilde \sigma_f$ is another such idempotent, we know that $\tilde \sigma_f= \sigma_f\gamma \delta'$
for a suitable $\gamma\in\Sm_{\R}(\Omega)\setminus\{0\}$ and $\delta'$ unitary quaternion, so that $\sigma_f^{c}*\mathfrak{b}\equiv 0$ and $\tilde \sigma_f^{c}*\mathfrak{b}\equiv 0$ are equivalent conditions.
\end{remark}

\begin{proof}[Proof of Theorem~\ref{singularcase}]
If $f$ is a zero divisor and there exists $\chi$ such that $\mathcal{L}_{f,g}(\chi)=\mathfrak{b}$, then $f*\chi*g=\mathfrak{b}$ and thus $f^{c}*\mathfrak{b}=f^{c}*f*\chi*g=f^{s}\chi*g\equiv 0$. 
Now write $f=2(f\delta)_{0}\sigma_f\delta^{c}$ for a suitable unitary quaternion $\delta$ and idempotent $\sigma_f$.
The equality $f^{c}*\mathfrak{b}=2(f\delta)_{0}\delta \sigma_f^{c}*\mathfrak{b}\equiv0$ implies $\sigma_f^{c}*\mathfrak{b}\equiv0$. As $\sigma_f^{c}=1-\sigma_f$ we obtain
$\mathfrak{b}=\sigma_f*\mathfrak{b}$.
Analogous considerations hold if $g$ is a zero divisor, showing that $\mathfrak{b}=\mathfrak{b}*\rho_g$.

Vice versa if $f$ is a zero divisor, $\mathfrak{b}=\sigma_f*\mathfrak{b}$ and $g$ is not a zero divisor, we have the following chain of equalities
\begin{align*}
\mathfrak{b}&=\sigma_f*\mathfrak{b}=\left[2(f\delta)_{0}\sigma_f\delta^{c}((2(f\delta)_{0})^{-1}\delta\right]*\mathfrak{b}*g^{-*}*g\\
&=f*\left[(2(f\delta)_{0})^{-1}\delta*\mathfrak{b}*g^{-*}\right]*g=\mathcal{L}_{f,g}((2(f\delta)_{0})^{-1}\delta*\mathfrak{b}*g^{-*}),
\end{align*}
which shows that $\mathcal{L}_{f,g}(\chi)=\mathfrak{b}$ admits a solution. If $f$ is not a zero divisor, $g$ is a zero divisor and 
$\mathfrak{b}=\mathfrak{b}*\rho_g$, the thesis follows by reasoning as before. 

If both $f$ and $g$ are zero divisors, $\mathfrak{b}=\sigma_f*\mathfrak{b}=\mathfrak{b}*\rho_g$, writing $f=2(f\delta)_{0}\sigma_f\delta^{c}$ and 
$g=2(g\eta)_{0} \eta^{c}*\rho_g$, the following chain of equalities yields the thesis
\begin{align*}
\mathfrak{b}&=\sigma_f*\mathfrak{b}=\left[2(f\delta)_{0}\sigma_f\delta^{c}((2(f\delta)_{0})^{-1}\delta\right]*\mathfrak{b}=f*\left[(2(f\delta)_{0})^{-1}\delta\right]*\mathfrak{b}\\
&=f*\left[(2(f\delta)_{0})^{-1}\delta\right]*\mathfrak{b}*\rho_g= f*\left[(2(f\delta)_{0})^{-1}\delta\right]*\mathfrak{b}*\left[(2(g\eta)_{0})^{-1}\eta*(g\eta)_{0}\eta^{c}\right]*\rho_g\\
&=f*\left[(2(f\delta)_{0})^{-1}\delta\right]*\mathfrak{b}*\left[(2(g\eta)_{0})^{-1}\eta\right]*g=\mathcal{L}_{f,g}\left(\left[(2(f\delta)_{0})^{-1}\delta\right]*\mathfrak{b}*\left[(2(g\eta)_{0})^{-1}\eta\right]\right).
\end{align*}
\end{proof}

We now describe the kernel of $\mathcal{L}_{f,g}$ when the operator is not an isomorphism.

\begin{proposition}\label{kernelfg}
Let $f,g\in\Sm(\Omega)\setminus\{0\}$ be such that $\mathcal{L}_{f,g}$ is not an isomorphism.
If $f$ is a zero divisor, for a suitable unitary $\eta\in\HH$, we denote by $\rho_{f}$ the idempotent given in formula~\eqref{zerodividemright}.
If $g$ is a zero divisor, for a suitable unitary $\delta\in\HH$, we denote by $\sigma_{g}$ the idempotent given in formula~\eqref{zerodividem}. 
Then $\chi\in\ker (\mathcal{L}_{f,g})$ if and only if 
\begin{enumerate}
\item $\rho_{f}*\chi\equiv 0$, if $f$ is a zero divisor and $g$ is not a zero divisor;
\item $\chi*\sigma_{g}\equiv 0$, if $g$ is a zero divisor and $f$ is not a zero divisor;
\item $\rho_{f}*\chi*\sigma_{g}\equiv 0$ if both $f$ and $g$ are zero divisors.
\end{enumerate}
\end{proposition}

\begin{proof}
\textit{(1)} As $g$ is not a zero divisor, then $\chi\in\ker(\mathcal{L}_{f,g})$ if and only if $f*\chi\equiv0$.
Choose a unitary quaternion $\eta$ such that $(f\eta)_{0}\not\equiv0$ and write
$f=2(f\eta)_{0}\eta^{c}*\rho_{f}$ as given in formula~\eqref{zerodividemright}. Now $f*\chi=2(f\eta)_{0}\eta^{c}*\rho_{f}*\chi\equiv 0$ is equivalent to $\rho_{f}*\chi\equiv 0$  since $(f\eta)_{0}\in\Sm_{\R}(\Omega)\setminus\{0\}$ and $\eta\neq0$.

\textit{(2)} This second case is obtained as in \textit{(1)} by using formula~\eqref{zerodividem}.

\textit{(3)} By definition $\chi\in\ker(\mathcal{L}_{f,g})$ if and only if $f*\chi*g\equiv0$. 
Choose two unitary quaternion $\delta$ and $\eta$ such that $(f\eta)_{0}\not\equiv0$, $(g\delta)_{0}\not\equiv 0$ and write
$f=2(f\eta)_{0}\eta^{c}*\rho_{f}$, as given in formula~\eqref{zerodividemright}, and $g=2(g\delta)_{0}\sigma_{g}\delta^{c}$ as in formula~\eqref{zerodividem}. Now $f*\chi*g=4(f\eta)_{0}(g\delta)_{0}\eta^{c}*\rho_{f}*\chi*\sigma_{g}\delta^{c}\equiv 0$ is equivalent to $\rho_{f}*\chi*\sigma_{g}\equiv 0$ since $(f\eta)_{0},(g\delta)_{0}\in\Sm_{\R}(\Omega)\setminus\{0\}$ and $\eta,\delta\neq0$.
\end{proof}


\section{Matrix representation of $\mathcal{L}_{\mathcal{F},\mathcal{G}}$-type equations}\label{representation}

The techniques used in the previous section to study the case $N=1$ are not powerful enough even to study the next step $N=2$.
To tackle the general case we need to represent the linear equations we are dealing with by means of square matrices in the 
same spirit of~\cite{janovska}.

Since we want to use coordinates for $\Sm(\Omega)$ over $\Sm_{\R}(\Omega)$, from now on we choose an orthonormal basis
$\mathcal{B}:=(1,I,J,K)$ of $\HH$ (which by Proposition~\ref{vectorspace} is a basis for $\Sm(\Omega)$ over $\Sm_{\R}(\Omega)$, too). Given
$f=f_{0}+f_{1}I+f_{2}J+f_{3}K$, we will denote by $F_{\mathcal{B}}:\Sm(\Omega)\to(\Sm_{\R}(\Omega))^{4}$ the usual
coordinates isomorphism 
$$
F_{\mathcal{B}}(f)=\begin{vmatrix}
f_{0}\\
f_{1} \\
f_{2}\\
f_{3}\\
\end{vmatrix}.
$$

\begin{definition}
For any $f=f_{0}+f_{1}I+f_{2}J+f_{3}K\in\Sm(\Omega)$ we define 
\begin{equation*}\label{matrices}
\imath_{L}(f):=\begin{vmatrix}
f_{0} & -f_{1} &-f_{2}&-f_{3}\\
f_{1} &f_{0}&-f_{3}&f_{2}\\
f_{2}&f_{3}&f_{0}&-f_{1}\\
f_{3}&-f_{2}&f_{1}&f_{0}
\end{vmatrix},\qquad\imath_{R}(f):=\begin{vmatrix}
f_{0} & -f_{1} &-f_{2}&-f_{3}\\
f_{1} &f_{0}&f_{3}&-f_{2}\\
f_{2}&-f_{3}&f_{0}&f_{1}\\
f_{3}&f_{2}&-f_{1}&f_{0}
\end{vmatrix}.
\end{equation*}

\end{definition}

\begin{lemma}
For any $f,g,h\in\Sm(\Omega)$, the following equalities hold.
\begin{align}
\imath_{R}(f*g)&=\imath_{R}(g)\imath_{R}(f),\nonumber\\
\imath_{L}(f)\imath_{R}(g)&=\imath_{R}(g)\imath_{L}(f),\nonumber\\
F_{\mathcal{B}}(f*g)&=\imath_{L}(f)F_{\mathcal{B}}(g)=\imath_{R}(g)F_{\mathcal{B}}(f),\nonumber\\
F_{\mathcal{B}}(f*g*h)&=\imath_{L}(f)\imath_{L}(g)F_{\mathcal{B}}(h)=\imath_{R}(h)\imath_{R}(g)F_{\mathcal{B}}(h),\nonumber\\
F_{\mathcal{B}}(f*g*h)&=\imath_{L}(f)\imath_{R}(h)F_{\mathcal{B}}(g)=\imath_{R}(h)\imath_{L}(f)F_{\mathcal{B}}(g),\label{propiLiR}\\
det(\imath_{L}(f))=det(\imath_{R}(f))&=(f^{s})^{2}.\label{determinantfs}
\end{align}
\end{lemma}

\begin{proof}
The proof of all equalities can be performed by direct inspection.
\end{proof}
Thanks to formula~\eqref{propiLiR},
for any two $N$-tuples $\mathcal{F}=(f_{[1]},\dots,f_{[N]}),\mathcal{G}=(g_{[1]},\dots,g_{[N]})\subset\Sm(\Omega)\setminus\{0\}$,
the linear operator $\mathcal{L}_{\mathcal{F},\mathcal{G}}$ given in formula~\eqref{elementary} can be written as
$$
F_{\mathcal{B}}(\mathcal{L}_{\mathcal{F},\mathcal{G}})(\chi)=\left(\sum_{n=1}^{N}\imath_{L}(f_{[n]})\imath_{R}(g_{[n]})\right)F_{\mathcal{B}}(\chi),
$$
and since $F_{\mathcal{B}}$ is an isomorphism, the solvability of $\mathcal{L}_{\mathcal{F},\mathcal{G}}(\chi)=\mathfrak{b}$
is equivalent to the solvability of $F_{\mathcal{B}}(\mathcal{L}_{\mathcal{F},\mathcal{G}})(\chi)=F_{\mathcal{B}}(\mathfrak{b})$.
This interpretation allows us to characterize the cases in which the operator $\mathcal{L}_{\mathcal{F},\mathcal{G}}$ 
is an isomorphism. 
\begin{proposition}\label{determinant}
The linear operator $\mathcal{L}_{\mathcal{F},\mathcal{G}}$ is an isomorphism if and only if 
$$
\det\left(\sum_{n=1}^{N}\imath_{L}(f_{[n]})\imath_{R}(g_{[n]})\right)\not\equiv 0.
$$
\end{proposition}

\begin{remark}
Last proposition gives a more algebraic interpretation of Proposition~\ref{isomorphism} \textit{(2)}. Indeed, when $N=1$
we have that $\mathcal{L}_{f,g}$ is an isomorphism if and only if $\det(\imath_{L}(f)\imath_{R}(g))=\det(\imath_{L}(f))\det(\imath_{R}(g))\not\equiv 0$. Thanks to formula~\eqref{determinantfs}, we have that
$$
\det(\imath_{L}(f))\det(\imath_{R}(g))=(f^{s})^{2}(g^{s})^{2},
$$
and the second term is identically zero if and only if either $f^{s}$ or $g^{s}$ vanish identically, which is the condition
that characterizes zero divisors and identically zero functions.
\end{remark}


From now on, we focus our attention on a specific class of $\mathcal{L}_{\mathcal{F},\mathcal{G}}$,
namely the cases when $N=2$, $\mathcal{F}=(f,1)$, $\mathcal{G}=(1,g)$.

\begin{definition}
Let $f,g\in\Sm(\Omega)$. The \textit{Sylvester operator} $\mathcal{S}_{f,g}$ associated to $f$ and $g$ is the $\Sm_{\R}(\Omega)$-linear operator given by
\begin{equation*}\label{sylvesterfunction}
\mathcal{S}_{f,g}(\chi):=\mathcal{L}_{(f,1),(1,g)}=f*\chi+\chi*g.
\end{equation*}
The associated \textit{Sylvester equation} with ``constant term'' $\mathfrak{b}$, is the $\Sm_{\R}(\Omega)$-linear equation given by
\begin{equation}\label{sylvesterequation}
\mathcal{S}_{f,g}(\chi)=\mathfrak{b}.
\end{equation}
\end{definition}

The name of ``Sylvester operator'' is due to the fact that, when dealing with matrices, equation~\eqref{sylvesterequation} is usually called
\textit{Sylvester equation}. 

\begin{remark}
In the case when $a_{1},a_{2},b_{1},b_{2}\in\HH\setminus\{0\}$, it is always possible to write the expression $a_{1}qb_{1}+a_{2}qb_{2}$ 
as $a_{2}(a_{2}^{-1}a_{1}q+qb_{2}b_{1}^{-1})b_{1}$ and then the solvability of $a_{1}qb_{1}+a_{2}qb_{2}=p$
is equivalent to the solvability of $(a_{2}^{-1}a_{1})q+q(b_{2}b_{1}^{-1})=a_{2}^{-1}pb_{1}^{-1}$, which is the Sylvester 
equation associated to $a_{2}^{-1}a_{1}$ and $b_{2}b_{1}^{-1}$.
In the case of slice (semi-)regular functions, the possible presence of zero divisors and the fact that the $*$-inverse of
a regular function is not always a regular function is an obstruction to the reduction of the general case to the Sylvester case.
\end{remark}

The following proposition shows that the Sylvester equation associated to $f$ and $g$ is also associated to a wider family of functions.

\begin{proposition}\label{propsameeq}
Let $f,g\in\Sm(\Omega)\setminus\{0\}$. Then for any $\alpha\in\Sm_{\R}(\Omega)$, we have
\begin{equation}\label{sameeq}
\mathcal{S}_{f,g}=\mathcal{S}_{f+\alpha,g-\alpha}.
\end{equation}
\end{proposition}
\begin{proof}
Indeed, for any $\chi\in\Sm(\Omega)$, we have
$$\mathcal{S}_{f+\alpha,g-\alpha}(\chi)=f*\chi+\alpha*\chi+\chi*g+\chi*(-\alpha)=f*\chi+\chi*g=\mathcal{S}_{f,g}(\chi),$$
since $\Sm_{\R}(\Omega)$ is the center of $\Sm(\Omega)$.
\end{proof}

We notice that, if $g_{v}\equiv 0$, then $\mathcal{S}_{f,g}=\mathcal{S}_{f+g_{0},0}=\mathcal{L}_{f+g_{0},1}$;
analogously, if $f_{v}\equiv 0$, then $\mathcal{S}_{f,g}=\mathcal{S}_{0,f_{0}+g}=\mathcal{L}_{1,f_{0}+g}$.
Since the operators of the class $\mathcal{L}_{f,g}$ were thoroughly studied in Section~\ref{lineq}, from now on,
without loss of generality, we shall work under the following 

\begin{assumption}\label{assumpfv}
We consider $\mathcal{S}_{f,g}$ where neither $f$ nor $g$ belong to $\Sm_{\R}(\Omega)$.
\end{assumption}

We now give two definitions that will be useful to study the invertibility of $\mathcal{S}_{f,g}$.

\begin{definition}
Let $f,g\in\Sm(\Omega)$. We say that $f$ and $g$ are equivalent and write $f\simeq g$ if there exists a $*$-invertible $h\in\Sm(\Omega)$, such that
$$f=h^{-*}*g*h.$$
\end{definition}

\begin{lemma}\label{necessary}
If $f\simeq g$, then $f_{0}\equiv g_{0}$ and $f^{s}\equiv g^{s}$ (that implies also $f_{v}^{s}\equiv g_{v}^{s}$).
In particular, if $f\simeq g$, then $f$ is a zero divisor if and only if $g$ is.
\end{lemma}

\begin{proof}
If we write $g=g_{0}+g_{v}$, we then have, for some invertible $h \in \mathcal{SEM}(\Omega)$,
$$
f=h^{-s}h^{c}*g*h=h^{-s}h^{c}*(g_{0}+g_{v})*h=h^{-s}h^{c}g_{0}h+h^{-s}h^{c}*g_{v}*h=g_{0}+h^{-s}h^{c}*g_{v}*h.
$$
Then, in order to prove that $f_{0}=g_{0}$, it is enough to show that $(h^{-s}h^{c}*g_{v}*h)_{0}\equiv 0$.
As $h^{-s}\in\Sm_{\R}(\Omega)$, we are left with showing that $(h^{c}*g_{v}*h)_{0}\equiv0$; indeed we have
$$
(h^{c}*g_{v}*h)^{c}=h^{c}*g_{v}^{c}*h=-h^{c}*g_{v}*h,
$$
and the equality $f_{0}=g_{0}$ is proven. The equality $f^{s}=g^{s}$ is now straightforward.

Last assertion follows immediately from the fact that $f$ is a zero divisor if and only if $f^{s}\equiv 0$ and
the same holds for $g$.
\end{proof}

An accurate study of the operator $\mathcal{S}_{f,g}$ will show that, if $f,g\not\in\Sm_{\R}(\Omega)$, then the equalities $f_{0}=g_{0}$ and $f_{v}^{s}=g_{v}^{s}$ imply
$f\simeq g$ (see Corollary~\ref{corollaryequivalence} if the domain is slice and Corollary~\ref{sufficiencyproduct} in the general case).

We now pass to the announced second definition.

\begin{definition}
Let $f,g\in\Sm(\Omega)$. We say that the couple $(f,g)$ \textit{intertwines with (a zero divisor) $\sigma$}, if there exists a zero divisor $\sigma$ such that $$f*\sigma=\sigma*g.$$ 
\end{definition}

\begin{example}
Let $\Omega$ be a product domain and choose $f$ and $g$ such that $f_{0}=g_{0}\not\equiv 0$, $f_{v}$ a zero divisor
and $g_{v}\equiv 0$, then we have $f_{v}^{s}=g_{v}^{s}\equiv 0$. 
We claim that $f\not\simeq g$ and that the couple of functions $(f,g)$ intertwines with an idempotent. 
Indeed, if there exists $h\in\Sm(\Omega)$ invertible such that $f=h^{-*}*g*h$, 
as $g_{v}\equiv 0$ we obtain $f\equiv g_{0}$, which contradicts the fact that $f_{v}$ is a zero divisor.
Now, write $f_{v}=2(f_{v}\eta)_{0} \eta^{c}*\rho$ for a suitable unitary $\eta\in\HH$ and $\rho$ idempotent 
as in equation~\eqref{zerodividemright}. Since $\rho*\rho^{c}\equiv 0$ and $g=g_{0}=f_{0}$, we have
$$f*\rho^{c}=(f_{0}+f_{v})*\rho^{c}=f_{0}\rho^{c}+2(f_{v}\eta)_{0} \eta^{c}*\rho*\rho^{c}=
f_{0}\rho^{c}=g_{0}\rho^{c}=\rho^{c}*g.$$
\end{example}

Next proposition characterizes the non-invertibility of $\mathcal{S}_{f,g}$ in terms of the previous definitions.

\begin{proposition}\label{isomorphism-sylv}
Given $f,g\in\Sm(\Omega)$, then $\mathcal{S}_{f,g}$ is not an isomorphism if and only if
one of the two following conditions holds
\begin{enumerate}
\item $f\simeq -g$;
\item there exist a zero divisor $\chi$ such that $(f,-g)$ intertwines with $\chi$.
\end{enumerate}
\end{proposition}
\begin{proof}
The operator $\mathcal{S}_{f,g}$ is not an isomorphism if and only if there exists $\chi\in\Sm(\Omega)\setminus\{0\}$
such that $f*\chi+\chi*g\equiv0$. If $\chi$ is not a zero divisor, then it is invertible in $\Sm(\Omega)$ and 
$-g=\chi^{-*}*f*\chi$ exactly means $f\simeq -g$. 
If $\chi$ is a zero divisor, then $f*\chi+\chi*g\equiv0$ exactly means that the couple $(f,-g)$ intertwines with $\chi$.
\end{proof}

Notice that the first condition says that there exists an invertible $\chi\in\ker(\mathcal{S}_{f,g})$, while the second one
means that a zero divisor belongs to $\ker(\mathcal{S}_{f,g})$. 

\begin{remark}\label{noproduct}
Trivially, if $\Omega$ is a slice domain, for any $f,g\in\Sm(\Omega)$, the kernel of $\mathcal{S}_{f,g}$ cannot contain 
zero divisors, so \textit{(2)}. can never take place and thus $\mathcal{S}_{f,g}$ is not an isomorphism if and only if $f\simeq-g$.
\end{remark}

Together with the previous remark, the following examples show that the two cases stated in Proposition~\ref{isomorphism-sylv} are not related.

\begin{example}\label{exampleintert}
Let $\Omega$ be a product domain and set 
$$
f=1-\mathcal{J}i,\qquad g=fj=(1-\mathcal{J}i)j=j-\mathcal{J}k.
$$
It is easily seen that $\chi=f^{c}\in\ker(\mathcal{S}_{f,g})$, while $f$ and $-g$ have different ``real parts'' and therefore,
thanks to Lemma~\ref{necessary}, they are not equivalent.
\end{example}

\begin{example}
Let $\sigma\in\So(\Omega)$ be an idempotent and set $f=\sigma, g=-\sigma$. 
Trivially any $\chi\in\Sm_{\R}(\Omega)$ belongs to $\ker{\mathcal{S}_{f,g}}$, as well as
$\chi=\sigma^{c}$.
\end{example}

\section{The rank of the Sylvester operator}\label{ranksylv}

We begin this section with a characterization of the invertibility of $\mathcal{S}_{f,g}$ by means of the matrix representation given
in Section~\ref{representation}. We recall that, by Assumption~\ref{assumpfv}, neither $f$ nor $g$ belong to $\Sm_{\R}(\Omega)$.
To simplify notation, from now on, we set
$$
S_{f,g}=\imath_{L}(f)+\imath_{R}(g).
$$

\begin{proposition}\label{propdetsyl}
Given $f=f_{0}+f_{v},g=g_{0}+g_{v}\in\Sm(\Omega)$,
the characteristic polynomial of the $\Sm_{\R}(\Omega)$-linear operator $\mathcal{S}_{f,g}$ is given by
\begin{equation*}\label{charpol}
p(\lambda)=(f_{0}+g_{0}-\lambda)^{2}[(f_{0}+g_{0}-\lambda)^{2}+2(f_{v}^{s}+g_{v}^{s})]+(f_{v}^{s}-g_{v}^{s})^{2}.
\end{equation*}
In particular $\mathcal{S}_{f,g}$ is an isomorphism if and only if 
\begin{equation}\label{determinantsylvester}
(f_{0}+g_{0})^{2}[(f_{0}+g_{0})^{2}+2(f_{v}^{s}+g_{v}^{s})]+(f_{v}^{s}-g_{v}^{s})^{2}\not\equiv 0.
\end{equation}
Moreover, the rank of $\mathcal{S}_{f,g}$ is always strictly greater than $1$.
\end{proposition}
\begin{proof}
First of all, given
$f=f_{0}+f_{1}i+f_{2}j+f_{3}k$ and $g=g_{0}+g_{1}i+g_{2}j+g_{3}k$, we write
\begin{equation}\label{matrixiLR}
S_{f,g}:=\imath_{L}(f)+\imath_{R}(g)=\begin{vmatrix}
f_{0} +g_{0}& -(f_{1}+g_{1}) &-(f_{2}+g_{2})&-(f_{3}+g_{3})\\
f_{1}+g_{1} &f_{0}+g_{0}&-(f_{3}-g_{3})&f_{2}-g_{2}\\
f_{2}+g_{2}&f_{3}-g_{3}&f_{0}+g_{0}&-(f_{1}-g_{1})\\
f_{3}+g_{3}&-(f_{2}-g_{2})&f_{1}-g_{1}&f_{0}+g_{0}
\end{vmatrix}.
\end{equation}
A long but straightforward computation gives
\begin{align}\label{charpolcomp}
p(\lambda)=\det (S_{f,g}-\lambda)=&(f_{0}+g_{0}-\lambda)^{4}+2(f_{0}+g_{0}-\lambda)^{2}(f_{1}^{2}+f_{2}^{2}+f_{3}^{2}+g_{1}^{2}+g_{2}^{2}+g_{3}^{2})\nonumber\\
&\quad+(f_{1}^{2}-g_{1}^{2}+f_{2}^{2}-g_{2}^{2}+f_{3}^{2}-g_{3}^{2})^{2}\nonumber\\
=&(f_{0}+g_{0}-\lambda)^{2}[(f_{0}+g_{0}-\lambda)^{2}+2(f_{v}^{s}+g_{v}^{s})]+(f_{v}^{s}-g_{v}^{s})^{2}\\
=&\lambda^{4}-\big[4(f_{0}+g_{0})\big]\lambda^{3}+\big[2(f_{v}^{s}+g_{v}^{s}+3(f_{0}+g_{0})^{2})\big]\lambda^{2}\nonumber\\
&\quad-\big[4(f_{0}+g_{0})((f_{0}+g_{0})^{2}+f_{v}^{s}+g_{v}^{s})\big]\lambda\nonumber\\
&\quad+(f_{0}+g_{0})^{2}[(f_{0}+g_{0})^{2}+2(f_{v}^{s}+g_{v}^{s})]+(f_{v}^{s}-g_{v}^{s})^{2}\nonumber
\end{align}
Thanks to Proposition~\ref{determinant}, we have that $\mathcal{S}_{f,g}$ is an isomorphism if and only if 
$\det(S_{f,g})=(f_{0}+g_{0})^{2}[(f_{0}+g_{0})^{2}+2(f_{v}^{s}+g_{v}^{s})]+(f_{v}^{s}-g_{v}^{s})^{2}\not\equiv0$. 

Suppose now that $\mathcal{S}_{f,g}$ has rank less than $2$. Then $\lambda=0$ is an eigenvalue of algebraic multiplicity at
least $3$, which gives
\begin{equation}\label{3eq}
\begin{cases}
f_{v}^{s}+g_{v}^{s}+3(f_{0}+g_{0})^{2}\equiv0\\
(f_{0}+g_{0})((f_{0}+g_{0})^{2}+f_{v}^{s}+g_{v}^{s})\equiv 0\\
(f_{0}+g_{0})^{2}[(f_{0}+g_{0})^{2}+2(f_{v}^{s}+g_{v}^{s})]+(f_{v}^{s}-g_{v}^{s})^{2}\equiv 0.
\end{cases}
\end{equation}
The second equation is equivalent to either $f_{0}+g_{0}\equiv 0$ or $(f_{0}+g_{0})^{2}+f_{v}^{s}+g_{v}^{s}\equiv 0$.
In the first case, since either $f_{v}+g_{v}$ or $f_{v}-g_{v}$ are not identically zero because of Assumption~\ref{assumpfv}, we can find a $2\times 2$ submatrix
of $S_{f,g}$ with determinant different from zero, which is a contradiction. In the second case, the first equation of system~\eqref{3eq}  together with $(f_{0}+g_{0})^{2}+f_{v}^{s}+g_{v}^{s}\equiv 0$ gives
$$
\begin{cases}
f_{v}^{s}+g_{v}^{s}+3(f_{0}+g_{0})^{2}\equiv0\\
(f_{0}+g_{0})^{2}+f_{v}^{s}+g_{v}^{s}\equiv 0,
\end{cases}
$$
which again entails $f_{0}+g_{0}\equiv 0$ and we are back to the previous contradiction.
%
%
%
%
%
%
%
%
\end{proof}

Last proposition allows us to prove that in the case of slice domains the relation $f\simeq g$ means exactly $f_{0}=g_{0}$
and $f^{s}\equiv g^{s}$. In fact this holds even for product domains, but the proof of this fact will require a much deeper
investigation on the kernel of $\mathcal{S}_{f,g}$.

\begin{corollary}\label{corollaryequivalence}
Let $f,g\in\Sm(\Omega)$ and $\Omega$ be a slice domain. Then $f\simeq g$ if and only if $f_{0}\equiv g_{0}$ and $f^{s}\equiv g^{s}$
(that is $f_{v}^{s}\equiv g_{v}^{s}$).
\end{corollary}

\begin{proof}
The necessity of the condition was shown in Lemma~\ref{necessary}. To prove its sufficiency, we notice that,
if $f_{0}\equiv g_{0}$ and $f^{s}\equiv g^{s}$, then $\det(S_{f,-g})\equiv 0$, hence
$\mathcal{S}_{f,-g}$ is not an isomorphism and therefore $\ker(\mathcal{S}_{f,-g})\neq\emptyset$. As $\Omega$ contains real points,
there are no zero divisors in $\Sm(\Omega)$ and therefore $\ker(\mathcal{S}_{f,-g})$ contains an invertible $\chi$,
which shows that $f\simeq g$.
\end{proof}



Next result gives a more precise characterization of the rank of $S_{f,g}$ when $f_{0}+g_{0}= 0$.

\begin{proposition}\label{ranksimeq}
Let $f,g\in\Sm(\Omega)$ be such that $f_{0}=-g_{0}$, then $\mbox{rk}(S_{f,g})=2$ if and only if
$f_{v}^{s}= g_{v}^{s}$. In particular if $f\simeq -g$, then $\mbox{rk}(S_{f,g})=2$.
\end{proposition}
\begin{proof}
Since $f_{0}=-g_{0}$, Proposition~\ref{propdetsyl} gives that $\mbox{rk}(S_{f,g})=4$ if and only if
$f_{v}^{s}\not\equiv g_{v}^{s}$. So we are left with computing the rank of $S_{f,g}$ when
$f_{v}^{s}=g_{v}^{s}$. The hypothesis $f_0=-g_0$ implies that $S_{f,g}$ is skew symmetric,
then it is enough to compute the determinants of the first $(m,n)$-minors $D_{m,n}$, with $1\leq m<n\leq4$.
Since \vspace{-0.4cm}
\begin{center}
\begin{tabular}{ccc}
$D_{1,2} = (f_{1}-g_{1})(f_{v}^{s}-g_{v}^{s})=0$\qquad & \qquad$ D_{1,3} = (g_{2}-f_{2})(f_{v}^{s}-g_{v}^{s})=0$\qquad & \qquad$D_{1,4} = (f_{3}-g_{3})(f_{v}^{s}-g_{v}^{s})=0$\\
 & & \\
$D_{2,3} = (f_{3}+g_{3})(f_{v}^{s}-g_{v}^{s})=0$\qquad & \qquad$ D_{2,4} = (f_{2}+g_{2})(f_{v}^{s}-g_{v}^{s})=0$\qquad & \qquad$D_{3,4} = (f_{1}+g_{1})(f_{v}^{s}-g_{v}^{s})=0$\\
\end{tabular}
\end{center}
then the rank of $S_{f,g}$ is less than or equal to $2$. As we proved in Proposition~\ref{propdetsyl} that the rank of 
$S_{f,g}$ is always strictly greater than $1$, we are done.
%
%
\end{proof}

We now give two examples in which $\mathcal{S}_{f,g}$ is not an isomorphism and $f_{0}+g_{0}\not\equiv0$.

\begin{example}\label{rank3ex1}
Let $\Omega$ be a product domain and set $f=\mathcal{J}i$ and $g=1+2\mathcal{J}k$.
Then $f_{0}+g_{0}=1$, $f_{v}^{s}\equiv -1$, $g_{v}^{s}\equiv -4$.
A direct computation shows that 
the characteristic polynomial in this case is equal to $\lambda^{4}-4\lambda^{3}-4\lambda^{2}+16\lambda$, thus $\lambda=0$ has algebraic multiplicity $1$ and $\mbox{rk}(S_{f,g})=3$.
\end{example}

\begin{example}\label{rank3ex2}
Let $\Omega$ be a product domain and define $f$ and $g$ as in Example~\ref{exampleintert}. Then 
$f_{0}=g_{2}=1$, $f_{1}=g_{3}=-\mathcal{J}$, $f_{2}=f_{3}=g_{0}=g_{1}\equiv 0$ and
hence $f_{v}^{s}=-1$, $g_{v}^{s}\equiv 0$. A direct computation shows that 
the characteristic polynomial is equal to $\lambda^{4}-4\lambda^{3}+4\lambda^{2}$, thus $\lambda=0$
has algebraic multiplicity $2$. Nonetheless a direct computation of $S_{f,g}$ shows that also in this case we have $\mbox{rk}(S_{f,g})=3$.
\end{example}

We underline that in both examples,  $\mbox{rk}(\mathcal{S}_{f,g})$ equals $3$; nonetheless in the first case the eigenvalue
$0$ has algebraic multiplicity equal to $1$, whilst in the second one it has algebraic multiplicity equal to $2$.
Inspired by these instances, we prove that if $\mathcal{S}_{f,g}$ is not an isomorphism and $f_{0}+g_{0}\not\equiv 0$, then  the rank of $\mathcal{S}_{f,g}$ is always equal to $3$.

\begin{theorem}\label{rank3}
Let $f,g\in\Sm(\Omega)$ be such that $\mathcal{S}_{f,g}$ is not an isomorphism. Then $f_{0}+g_{0}\not\equiv 0$ if and only if $\mathcal{S}_{f,g}$ has rank $3$.
\end{theorem}
\begin{proof}
If $f_{0}+g_{0}\equiv 0$ we already proved that the rank of $\mathcal{S}_{f,g}$ is equal to $2$. 

Now
suppose that $f_{0}+g_{0}\not\equiv 0$ and consider the characteristic polynomial of $S_{f,g}$.
If $0$ is an eigenvalue of algebraic multiplicity $1$, then trivially the rank of $\mathcal{S}_{f,g}$ is equal to $3$.

Therefore we are left with dealing with the case in which $0$ is an eigenvalue of algebraic multiplicity at least~$2$, which by
formula~\eqref{charpolcomp} and $f_{0}+g_{0}\not\equiv 0$ yields
$$
\begin{cases}
(f_{0}+g_{0})^{2}+f_{v}^{s}+g_{v}^{s}\equiv 0\\
(f_{0}+g_{0})^{2}[(f_{0}+g_{0})^{2}+2(f_{v}^{s}+g_{v}^{s})]+(f_{v}^{s}-g_{v}^{s})^{2}\equiv 0,
\end{cases}
$$
which is equivalent to 
\begin{equation}\label{system1}
\begin{cases}
(f_{0}+g_{0})^{2}+f_{v}^{s}+g_{v}^{s}\equiv 0\\
f_{v}^{s}g_{v}^{s}\equiv 0.
\end{cases}
\end{equation}
Since $\Sm_{\R}(\Omega)$ is a field, then either $f_{v}^{s}$ or $g_{v}^{s}$ is identically zero. We perform
the computation in the first case, the second one being completely analogous. Thus System~\eqref{system1}
gives
\begin{equation*}\label{condition1}
f_{v}^{s}\equiv 0\qquad\mbox{and}\qquad(f_{0}+g_{0})^{2}+g_{v}^{s}\equiv 0.
\end{equation*}
Since $\mbox{rk}(\mathcal{S}_{f,g})=3$ if and only if the cofactor matrix of $S_{f,g}$ is not identically zero,
we suppose by contradiction that $\mbox{cof}(S_{f,g})=0$ which in particular implies $\mbox{cof}(S_{f,g})+\mbox{cof}(S_{f,g})^{T}=0$.
Up to a factor $2(f_{0}+g_{0})\not\equiv 0$, the elements of this matrix in positions $(1,2)$, $(1,3)$ and $(1,4)$
give the following system of equalities
\begin{equation*}\label{system6eq}
\begin{cases}
g_{3}f_{2}-f_{3}g_{2}\equiv 0\\
g_{1}f_{3}-f_{1}g_{3}\equiv 0\\
g_{1}f_{2}-f_{1}g_{2}\equiv 0,
\end{cases}
\end{equation*}
which means $f_{v}\pv g_{v}\equiv 0$. By~\cite[Proposition 2.10]{A-dF} this entails that $f_{v}$ and $g_{v}$ are linearly dependent over $\Sm_{\R}(\Omega)$. Nonetheless $f_{v}^{s}\equiv 0$ and $g_{v}^{s}=-(f_{0}+g_{0})^{2}\not\equiv 0$. As $f_{v}\not\equiv 0$, this is a contradiction which shows that $\mbox{rk}(\mathcal{S}_{f,g})=3$.
\end{proof}

\begin{remark}\label{rank3symmetric}
Notice that, the fact that $\mathcal{S}_{f,g}$ has rank 3 is symmetric in $f$ and $g$.
Indeed, Proposition~\ref{propdetsyl}, via Formula~\eqref{determinantsylvester}, guarantees that $\mathcal{S}_{f,g}$
is an isomorphism if and only if $\mathcal{S}_{g,f}$ is. Now it is enough to highlight that the condition on the sum of the ``real parts''
given in Theorem~\ref{rank3} is symmetric.
\end{remark}

\section{The solution of the Sylvester equation in the non-singular case}

In this section, we study the case in which $\mathcal{S}_{f,g}$ is an isomorphism, looking for the solution of the Sylvester equation $\mathcal{L}_{f,g}(\chi)=\mathfrak{b}$, given $f,g,\mathfrak{b}\in\Sm(\Omega)$. 
Some of the tools we introduce are inspired by the work of Bolotnikov~\cite{bolotnikov1,bolotnikov2}.

First of all, we notice that Proposition~\ref{propsameeq} allows us to consider the Sylvester equation only in the cases
in which neither $f$ nor $g$ are zero divisors, as a consequence of the following

\begin{lemma}\label{lemmanozerodiv}
For any $f,g\in\Sm(\Omega)$ there exists $\alpha\in\mathbb{R}$ such that neither $f+\alpha$ nor $g-\alpha$ are zero divisors.
\end{lemma}
\begin{proof}
If neither $f$ nor $g$ are zero divisors, we can take $\alpha\equiv 0$. If $f$ is a zero divisor, then $f^{s}=f_{0}^{2}+f_{v}^{s}\equiv 0$.
Now $(f+\alpha)^{s}=2\alpha f_{0}+\alpha^{2}=\alpha(2f_{0}+\alpha)\equiv 0$ if and only if either $\alpha\equiv0$ or $\alpha\equiv-\frac{f_{0}}{2}$. Since $(g-\alpha)^{s}=\alpha^{2}-2g_{0}\alpha+g^{s}$, it is enough to choose $\alpha$ any real number
such that $\alpha\neq 0$, $\alpha\not\equiv -\frac{f_{0}}{2}$ and $\alpha^{2}-2g_{0}\alpha+g^{s}\not\equiv 0$ to obtain 
that neither $f+\alpha$ nor $g-\alpha$ are zero divisors.
\end{proof}

Notice that Lemma~\ref{lemmanozerodiv} and equality~\eqref{sameeq} only deal with ``real parts'' of the functions $f$ and $g$, while Assumption~\ref{assumpfv} only deals with their ``vectorial parts'', so that they are independent.

\begin{assumption}\label{assumptionzero}
Without any loss of generality, in this section we shall consider only Sylvester operators associated to functions $f,g\not\in\Sm_{\R}(\Omega)$ none of which is a zero divisor.
\end{assumption}

We now define two functions $\lambda_L, \lambda_R\in \Sm(\Omega)$ which will be used to write explicitly the solution of $\mathcal{S}_{f,g}(\chi)=\mathfrak{b}$ when $\mathcal{S}_{f,g}$ is an isomorphism (see Theorem~\ref{explicit-solution}).

\begin{definition}\label{lambdaLR}
Let $f=f_{0}+f_{v},g=g_{0}+g_{v}\in\Sm(\Omega)$. If $f$ is not a zero divisor, we define $\lambda_{L}\in\Sm(\Omega)$, as
\begin{equation*}\label{lambdaL}
\lambda_{L}:=2g_{0}+f+g^{s}f^{-*}.
\end{equation*}
If $g$ is not a zero divisor, we define $\lambda_{R}\in\Sm(\Omega)$, as
\begin{equation*}\label{lambdaR}
\lambda_{R}:=2f_{0}+g+f^{s}g^{-*}.
\end{equation*}
\end{definition}

Notice that, if $f$ is not a zero divisor, then $\lambda_{L}\equiv 0$ if and only if $\lambda_{L}*f\equiv 0$ if and only if $f^{*2}+2g_{0}f+g^{s}\equiv 0$. Analogously, if $g$ is not a zero divisor, then $\lambda_{R}\equiv 0$ if and only if $g^{*2}+2f_{0}g+f^{s}\equiv 0$.

\begin{proposition}\label{alphaLR1}
Let $f,g\in\Sm(\Omega)$ be such that $f\simeq -g$. If $f$ (and then $g$) is not a zero divisor, then
$\lambda_{L}=\lambda_{R}\equiv 0$.
\end{proposition}

\begin{proof}
Thanks to Lemma~\ref{necessary}, we know that $f$ is a zero divisor if and only if $-g$ is; moreover, $f_{0}\equiv -g_{0}$
and $f_{v}^{s}\equiv g_{v}^{s}$. 

If $f$ is not a zero divisor,
then $\lambda_{L}\equiv 0$ if and only if $f^{*2}+2g_{0}f+g^{s}\equiv 0$. The following chain of equalities yields that $\lambda_L \equiv 0$:
$$
f^{*2}+2g_{0}f+g^{s}=f_{0}^{2}-f_{v}^{s}+2f_{0}f_{v}+2g_{0}f_{0}+2g_{0}f_{v}+g_{0}^{2}+g_{v}^{s}=(f_{0}+g_{0})^{2}+2(f_{0}+g_{0})f_{v}+g_{v}^{s}-f_{v}^{s}\equiv 0.
$$
The equality $\lambda_R \equiv 0$ follows by similar computations.
\end{proof}

We now give a partial converse of the previous proposition.

\begin{proposition}
Let $\Omega$ be a slice domain and $f=f_{0}+f_{v},g=g_{0}+g_{v}\in\Sm(\Omega)\setminus\{0\}$. Then $f\simeq-g$ if and only if $\lambda_{L}\equiv0$
if and only if $\lambda_{R}\equiv 0$.
\end{proposition}

\begin{proof}
First of all notice that, being $\Omega$ a slice domain and $f,g\not\equiv 0$, both $\lambda_{L}$ and $\lambda_{R}$
are well defined. Thanks to Proposition~\ref{alphaLR1}, we are left with proving that $\lambda_{L}\equiv 0$ implies $f\simeq -g$.
If $\lambda_{L}\equiv 0$, we have that $f^{*2}+2g_{0}f+g^{s}\equiv 0$. Last quantity can also be written as
$f_{0}^{2}-f_{v}^{s}+2f_{0}f_{v}+2g_{0}f_{0}+2g_{0}f_{v}+g_{0}^{2}+g_{v}^{s}$ and hence, by splitting in ``real'' and ``vector'' parts, we obtain the following system of equations
\begin{equation}\label{systemgammaL}
\begin{cases}
f_{0}^{2}-f_{v}^{s}+2g_{0}f_{0}+g_{0}^{2}+g_{v}^{s}\equiv 0\\
2(f_{0}+g_{0})f_{v}\equiv 0.
\end{cases}
\end{equation}
Since $\Sm_{\R}(\Omega)$ is a field, the second equation is satisfyied if and only if either $f_{0}\equiv -g_{0}$ or $f_{v}\equiv 0$.
If $f_{0}\equiv -g_{0}$, the first equation of system~\eqref{systemgammaL} becomes $-f_{v}^{s}+g_{v}^{s}\equiv 0$, that is
$g_{v}^{s}\equiv f_{v}^{s}$ and corollary~\ref{corollaryequivalence} entails $f\simeq -g$.
If $f_{0}+g_{0}\not\equiv 0$, then $f_{v}\equiv 0$. The first equation of system~\eqref{systemgammaL} then becomes
$(f_{0}+g_{0})^{2}+g_{v}^{s}\equiv 0$ which is a contradiction to the fact that $\Omega$ contains real points, where $(f_{0}+g_{0})^{2}\geq0$, $g_{v}^{s}\geq 0$ and $(f_{0}+g_{0})^{2}=0$ only occurs on a discrete set.\end{proof}

If $\mathcal{S}_{f,g}$ is an isomorphism we are now able to write explicitly the solution of $\mathcal{S}_{f,g}(\chi)=\mathfrak{b}$. 
Recall that, by Assumption~\ref{assumptionzero}, neither $f$ nor $g$ are zero divisors.

\begin{theorem}\label{explicit-solution}
Let $f,g\in\Sm(\Omega)$ be such that $\mathcal{S}_{f,g}$ is an isomorphism.
Then for any
$\mathfrak{b}\in\Sm(\Omega)$, the unique solution of $\mathcal{S}_{f,g}(\chi)=\mathfrak{b}$ is given by
$$
\chi=\lambda_{L}^{-*}*(\mathfrak{b}+f^{-*}*\mathfrak{b}*g^{c})=(\mathfrak{b}+f^{c}*\mathfrak{b}*g^{-*})*\lambda_{R}^{-*},
$$
where $\lambda_{L}$ and $\lambda_{R}$ are given by Definition~\ref{lambdaLR}.
\end{theorem}

\begin{proof} 
As $f$ and $g$ are not zero divisors, then both $\lambda_{L}$ and $\lambda_{R}$ are well defined.
We now prove that both $\lambda_{L}^{s}$ and $\lambda_{R}^{s}$ are not identically zero.
Since $f$ is not a zero divisor, then $\lambda_{L}$ is invertible if and only if $\lambda_{L}^{s}\not\equiv 0$ if and only if $(f*\lambda_{L})^{s}\not\equiv 0$.
Now we have 
\begin{align*}
(f*\lambda_{L})^{s}=&(2g_{0}f +f^{*2}+g^{s})^{s}=4g_{0}^{2}f^{s}+ f^{2s}+g^{2s}+4g_{0}\langle f,f^{*2}\rangle_{*}+4g_{0}g^{s}f_{0}+2g^{s}(f^{*2})_{0}\\
=&4g_{0}^{2}f_{0}^{2}+ 4g_{0}^{2}f_{v}^{s}+f_{0}^{4}+2f_{0}^{2}f_{v}^{s}+ g_{0}^{4}+2g_{0}^{2}g_{v}^{s}+4g_{0}f_{0}^{3}+4g_{0}f_{0}f_{v}^{s}\\
&  \quad\quad\,\,\,+4g_{0}^{3}f_{0}+4f_{0}g_{0}g_{v}^{s} +2g_{0}^{2}f_{0}^{2}-2g_{0}^{2}f_{v}^{s}+2f_{0}^{2}g_{v}^{s}+(f_{v}^{s})^{2}-2f_{v}^{s}g_{v}^{s}
+(g_{v}^{s})^{2}\\
=&(f_{0}+g_{0})^{4} +2[(g_{0}^{2}+f_{0}^{2}+2f_{0}g_{0})f_{v}^{s} +(g_{0}^{2}+2g_{0}f_{0}+f_{0}^{2})g_{v}^{s}]
+(f_{v}^{s}-g_{v}^{s})^{2}\\
=&(f_{0}+g_{0})^{2}[(f_{0} +g_{0})^{2}+2(f_{v}^{s}+g_{v}^{s})] +(f_{v}^{s}-g_{v}^{s})^{2}.
\end{align*}
As $\mathcal{S}_{f,g}$ is an isomorphism, by Proposition~\ref{propdetsyl} we have that last term is not identically zero and
hence $\lambda_{L}$ is invertible. 
An analogous computation gives that $\lambda_{R}$ is invertible.


Now, for any $\chi\in\Sm(\Omega)$ we have the following chain of equalities
\begin{align*}
f^{-*}*\mathcal{S}_{f,g}(\chi)*g^{c}+\mathcal{S}_{f,g}(\chi)&=f^{-*}*(f*\chi+\chi*g)*g^{c}+f*\chi+\chi*g\\
&=\chi*g^{c}+f^{-*}*\chi*g^{s}+f*\chi+\chi*g\\
&=\chi(g+g^{c})+g^{s}f^{-*}*\chi+f*\chi\\
&=2g_{0}\chi+g^{s}f^{-*}*\chi+f*\chi=(2g_{0}+g^{s}f^{-*}+f)*\chi=\lambda_{L}*\chi.
\end{align*}
Therefore, if $\chi$ is the solution of $\mathcal{S}_{f,g}(\chi)=\mathfrak{b}$, we obtain $f^{-*}*\mathfrak{b}*g^{c}+\mathfrak{b}=\lambda_{L}*\chi$,
which gives 
$$\chi=\lambda_{L}^{-*}*(f^{-*}*\mathfrak{b}*g^{c}+\mathfrak{b}).$$
The second equality of the statement is obtained analogously.
\end{proof}

\section{Sylvester operators of rank $2$}\label{rank2section}
We now consider the case when the Sylvester operator $\mathcal{S}_{f,g}$ has rank 2;  by Proposition~\ref{ranksimeq} and Theorem~\ref{rank3} this means exactly that $f_{0}=-g_{0}$ and $f_{v}^{s}=g_{v}^{s}$ (we recall that, by Assumption~\ref{assumpfv}, both $f_{v}$ and $g_{v}$ are not identically zero). Next statement describes the kernel of $\mathcal{S}_{f,g}$ under the conditions $f_{0}=-g_{0}$ and $f_{v}^{s}=g_{v}^{s}$.

\begin{theorem}\label{kerthm}
Let $f,g\in\Sm(\Omega)$ be such that $f_{0}=-g_{0}$ and $f_{v}^{s}=g_{v}^{s}$. Then
\begin{equation}\label{ker}
\ker(\mathcal{S}_{f,g})=\{f*h+h*g^{c}\,|\,h\in\Sm(\Omega)\}.
\end{equation}
Moreover, it is possible to find a basis of $\ker(\mathcal{S}_{f,g})$ consisting of invertible elements.
\end{theorem}
\begin{proof}
Notice that, since $f_{0}=-g_{0}$, for any $h\in\Sm(\Omega)$ we have $\mathcal{S}_{f,g}=\mathcal{S}_{f_{v},g_{v}}$ and $f*h+h*g^{c}=f_{v}*h-h*g_{v}$. Then
\begin{align*}
\mathcal{S}_{f,g}(f_{v}*h-h*g_{v})&=f_{v}*(f_{v}*h-h*g_{v})-(f_{v}*h-h*g_{v})*g_{v}\\
&=-f_{v}^{s}*h-f_{v}*h*g_{v}+f_{v}*h*g_{v}+h*g_{v}^{s}\equiv 0.
\end{align*}
The hypotheses on $f$ and $g$ together with Proposition~\ref{ranksimeq} guarantee that in order
to prove the equality of the two subspaces in formula~\eqref{ker}, it is enough to show that the $\Sm_{\R}(\Omega)$-linear subspace $\{f_{v}*h-h*g_{v}\,|\,h\in\Sm(\Omega)\}$ has dimension at least 2.
If $h=h_{0}+h_{v}$ we have 
\begin{align*}
f_{v}*h-h*g_{v}&=h_{0}(f_{v}-g_{v})-\langle f_{v},h_{v}\rangle_{*}+f_{v}\pv h_{v}+\langle g_{v}, h_{v}\rangle_{*}-h_{v}\pv g_{v}\\
&=\langle g_{v}-f_{v},h_{v}\rangle_{*}+\left[h_{0}(f_{v}-g_{v})+(f_{v}+g_{v})\pv h_{v}\right],
\end{align*}
where the first summand belongs to $\Sm_{\R}(\Omega)$ and the second has ``real part'' equal to zero.
If $f_{v}\neq g_{v}$ we take $\delta\in\SF$ such that $\langle g_{v}-f_{v}, \delta\rangle_{*}\not\equiv 0$.
Then $f_{v}*1-1*g_{v}$ and $f_{v}*\delta-\delta*g_{v}$ are linearly independent since the first has ``real part'' equal to zero
and it is not identically zero, while the second has ``real part'' equal to $\langle g_{v}-f_{v}, \delta\rangle_{*}\not\equiv 0$.
If $f_{v}=g_{v}$, we have $f_{v}*h-h*g_{v}=2f_{v}\pv h_{v}$. As $f_{v}\not\equiv 0$, we can find
two imaginary units $I,J\in\SF$, such that $2f_{v}\pv I$ and $2f_{v}\pv J$ are linearly independent, showing that
$\{f_{v}*h-h*g_{v}\,|\,h\in\Sm(\Omega)\}$ has dimension at least $2$ and thus proving equality~\eqref{ker}.

We now prove the existence of a basis of invertible elements. We start by computing explicitly $(f_{v}*h-h*g_{v})^{s}$;
for $h\in\Sm(\Omega)$ we have
$$(f_{v}*h-h*g_{v})^{s}=f_{v}^{s}h^{s}+g_{v}^{s}h^{s}-2\langle f_{v}*h,h*g_{v}\rangle_{*}= 2(f_{v}^{s}h^{s}-\langle f_{v}*h,h*g_{v}\rangle_{*}).
$$
For any unitary $\delta\in\HH$, we set $h\equiv \delta$ and find 
\begin{equation*}\label{basiseq}
(f_{v}*\delta-\delta*g_{v})^{s}=2(f_{v}^{s}-\langle f_{v}*\delta,\delta*g_{v}\rangle_{*})=2(f_{v}^{s}-\langle f_{v},\delta*g_{v}*\delta^{c}\rangle_{*}).
\end{equation*}
First of all we want to show that there exists an invertible element in $\ker(\mathcal{S}_{f,g})$. Indeed, if this is not, we have that
$(f_{v}*\delta-\delta*g_{v})^{s}\equiv 0$ for any unitary $\delta\in\HH$. In particular, choosing $\delta=1,i,j,k$, we obtain
$$
\begin{cases}
f_{v}^{s}\equiv \langle f_{v},g_{v}\rangle_{*}\equiv f_{1}g_{1}+f_{2}g_{2}+f_{3}g_{3}\\
f_{v}^{s}\equiv\langle f_{v},-i*g_{v}*i\rangle_{*}\equiv f_{1}g_{1}-f_{2}g_{2}-f_{3}g_{3}\\
f_{v}^{s}\equiv\langle f_{v},-j*g_{v}*j\rangle_{*}\equiv -f_{1}g_{1}+f_{2}g_{2}-f_{3}g_{3}\\
f_{v}^{s}\equiv\langle f_{v},-k*g_{v}*k\rangle_{*}\equiv -f_{1}g_{1}-f_{2}g_{2}+f_{3}g_{3}.
\end{cases}
$$
Adding up all four equations we find $f_{v}^{s}(=g_{v}^{s})\equiv 0$. Adding up the first equation with the second, third and fourth one,
we find $f_{1}g_{1}\equiv 0$, $f_{2}g_{2}\equiv 0$ and $f_{3}g_{3}\equiv 0$. Since $\Sm_{\R}(\Omega)$ is a field, at least one
between $f_{v}$ and $g_{v}$ has two components which are identically zero. This, together with $f_{v}^{s}(=g_{v}^{s})\equiv 0$,
implies that either $f_{v}\equiv 0$ or $g_{v}\equiv 0$, contradicting Assumption~\ref{assumpfv}.

Since we found an invertible element $\tau_{1}\in\ker(\mathcal{S}_{f,g})$ we can complete it to a basis $(\tau_{1},\tau_{2})$.
If both $\tau_{1}$ and $\tau_{2}$ are invertible, we are done. Otherwise consider the following linear combination: $\alpha\tau_{1}+\tau_{2}$ which is linearly independent from $\tau_{1}$ for any $\alpha\in\Sm_{\R}(\Omega)$.
We have 
$$(\alpha\tau_{1}+\tau_{2})^{s}=\alpha^{2}\tau_{1}^{s}+2\alpha\langle \tau_{1},\tau_{2}\rangle_{*}=\alpha(\alpha\tau_{1}^{s}+2\langle \tau_{1},\tau_{2}\rangle_{*}).$$ 
Therefore it is enough to chose $\alpha\not\equiv 0$ and $\alpha\not\equiv 2\tau_{1}^{-s}\langle \tau_{1},\tau_{2}\rangle_{*}$ to obtain the required basis.
\end{proof}

The full strength of Theorem~\ref{kerthm} discloses in the following corollary which states that two functions $f,g\in\Sm(\Omega)\setminus\Sm_{\R}(\Omega)$
are equivalent if and only if $f_{0}\equiv g_{0}$ and $f_{v}^{s}\equiv g_{v}^{s}$. 
Indeed, the existence of an invertible element in $\ker(\mathcal{S}_{f,g})$ implies that $f$ and $g$ are equivalent; thus an operatorial result is applied to function  theory in order to give a necessary and sufficient condition for the equivalence of a couple of slice semi-regular functions (compare with Lemma~\ref{necessary} which contains the necessary condition). 

\begin{corollary}\label{sufficiencyproduct}
Let $f,g\in\Sm(\Omega)\setminus\Sm_{\R}(\Omega)$ be such that $f_{0}\equiv g_{0}$ and $f_{v}^{s}\equiv g_{v}^{s}$. Then $f\simeq g$.
\end{corollary}

\begin{proof}
Consider the operator $\mathcal{S}_{f,-g}$. Theorem~\ref{kerthm}
guarantees the existence of an invertible $h\in\ker(\mathcal{S}_{f,-g})$, that is $\mathcal{S}_{f,-g}(h)=f*h-h*g\equiv0$.
This equality can also be written as $h^{-*}*f*h=g$, i.e. $f\simeq g$.
\end{proof}

Under suitable hypotheses, it is possible to describe $\ker(\mathcal{S}_{f,g})$ in a simpler way.

\begin{corollary}\label{specialbasis}
Let $f,g\in\Sm(\Omega)$ be such that $f\simeq -g$ and $(f_{v}-g_{v})^{s}\not\equiv 0$. Then
$$
\ker(\mathcal{S}_{f,g})= \mbox{Span} _{\Sm_{\R}(\Omega)}(f_{v}-g_{v},f_{v}^{s}+g_{v}^{s}+2f_{v}*g_{v}).
$$
\end{corollary}

\begin{proof}
As $f_{v}-g_{v}=f*1+1*g^{c}$ and $f_{v}^{s}+g_{v}^{s}+2f_{v}*g_{v}=2f_{v}^{s}+2f_{v}*g_{v}=f*(-2f_{v})+(-2f_{v})*g^{c}$, we have that 
$$\mbox{Span}_{\Sm_{\R}(\Omega)}(f_{v}-g_{v},f_{v}^{s}+g_{v}^{s}+2f_{v}*g_{v})\subseteq \ker(\mathcal{S}_{f,g}).$$
To show the equality it is sufficient to prove that $f_{v}-g_{v},f_{v}^{s}+f_{v}*g_{v}$ are linearly independent.
Since $f_{v}-g_{v}\not\equiv 0$ has zero ``real part'' and $f_{v}^{s}+g_{v}^{s}+2f_{v}*g_{v}=f_{v}^{s}+g_{v}^{s}-2\langle f_{v},g_{v}\rangle_{*}+2f_{v}\pv g_{v}$
has ``real part'' equal to $2(f_{v}^{s}-\langle f_{v},g_{v}\rangle_{*})=(f_{v}-g_{v})^{s}\not\equiv 0$, then we are done.
\end{proof}

The above result allows us to understand under which conditions on $f$ and $g$, the kernel of  $\mathcal{S}_{f,g}$ contains a zero divisor; obviously what follows is of interest only if $\Omega$ is a product domain.

\begin{proposition}
Let $f,g\in\Sm(\Omega)\setminus\Sm_{\R}(\Omega)$ be such that $f\simeq -g$. Then
$\ker(\mathcal{S}_{f,g})$ contains a zero divisor if and only if one of the following conditions holds
\begin{enumerate}
\item $f_{v}=g_{v}$ and $f_{v}^{s}$ has a square root;
\item $f_{v}\neq g_{v}$ and $(f_{v}-g_{v})^{s}\equiv 0$;
\item  $(f_{v}-g_{v})^{s}\not\equiv 0$ and $f_{v}^{s}$ has a square root.
\end{enumerate}
\end{proposition}
\begin{proof}
If $f_{v}=g_{v}$ then $\ker(\mathcal{S}_{f,g})=\ker(\mathcal{S}_{f_{v},f_{v}})=\{f_{v}*h-h*f_{v}\,|\,h\in\Sm(\Omega)\}
=\{f_{v}\pv h_{v}\,|\,h\in\Sm(\Omega)\}$.
Since $f_{v}\not\equiv 0$, we can choose an orthonormal basis $(1,I,J,K)\subset\HH$ such that $f_{1}\not\equiv 0$.
Thus a basis of $\ker(\mathcal{S}_{f,g})$ is given by $f_{v}\pv J=-f_{3}I+f_{1}K$ and $f_{v}\pv K=f_{2}I-f_{1}J$.
Now suppose that $\ker(\mathcal{S}_{f,g})$ contains a zero divisor. If $f_{v}\pv J$ is a zero divisor, then $f_{1}^{2}+f_{3}^{2}\equiv 0$ and hence $f_{v}^{s}=f_{1}^{2}+f_{2}^{2}+f_{3}^{2}=f_{2}^{2}$ has a square root. If $f_{v}\pv J$ is not a zero divisor, then
there exists $\alpha\in\Sm_{\R}(\Omega)$ such that $\alpha(f_{v}\pv J)+f_{v}\pv K$ is a zero divisor which can also be written as
 
$$
0\equiv (\alpha(f_{v}\pv J)+f_{v}\pv K)^{s}=((f_{2}-\alpha f_{3})I-f_{1}J+\alpha f_{1}K)^{s}=\alpha^{2}(f_{1}^{2}+f_{3}^{2})-2\alpha f_{2}f_{3}+f_{2}^{2}+f_{1}^{2}.
$$
By multiplying last term by $f_{1}^{2}+f_{3}^{2}$ we
equivalently obtain $(\alpha(f_{1}^{2}+f_{3}^{2})-f_{2}f_{3})^{2}=-f_{1}^{2}(f_{1}^{2}+f_{2}^{2}+f_{3}^{2})=(\mathcal{J}f_{1})^{2}f_{v}^{s}$, showing that $f_{v}^{s}$ has a square root.
Vice versa, if $f_{1}^{2}+f_{3}^{2}\equiv 0$, then $f_{v}\pv J$ is a zero divisor which belongs to $\ker(\mathcal{S}_{f,g})$.
Otherwise, if $f_{1}^{2}+f_{3}^{2}\not\equiv 0$ and $f_{v}^{s}$ has a square root $\rho$, a long but straightforward computation
of its symmetrized function shows that 
$$
(f_{2}f_{3}+\mathcal{J}f_{1}\rho)f_{v}\pv J+(f_{1}^{2}+f_{3}^{2})f_{v}\pv K
$$
is a zero divisor which belongs to $\ker(\mathcal{S}_{f,g})$.

Now assume $f_{v}\neq g_{v}$. If $(f_{v}-g_{v})^{s}\equiv 0$, then $f_{v}-g_{v}$ is a zero divisor which belongs to $\ker(\mathcal{S}_{f,g})$.

Finally, if $(f_{v}-g_{v})^{s}\not\equiv 0$, Corollary~\ref{specialbasis} states that $f_{v}-g_{v}, f_{v}^{s}+f_{v}*g_{v}$ is a basis
of $\ker(\mathcal{S}_{f,g})$. Then, there exists $\alpha\in\Sm_{\R}(\Omega)$ such that $\alpha(f_{v}-g_{v})+f_{v}^{s}+f_{v}*g_{v}$ is a zero divisor if and only if $(\alpha(f_{v}-g_{v})+f_{v}^{s}+f_{v}*g_{v})^{s}\equiv 0$.
We first compute $\langle f_{v}-g_{v},  f_{v}^{s}+f_{v}*g_{v}\rangle_{*}$. Since $f_{v}-g_{v}$ has no ``real part'', we have
$$
\langle f_{v}-g_{v},  f_{v}^{s}+f_{v}*g_{v}\rangle_{*}=\langle f_{v}-g_{v},  f_{v}^{s}-\langle f_{v},g_{v}\rangle_{*} +f_{v}\pv g_{v}\rangle_{*}= \langle f_{v}-g_{v}, f_{v}\pv g_{v}\rangle_{*}\equiv 0.
$$
As a consequence we obtain
$$
(\alpha(f_{v}-g_{v})+f_{v}^{s}+f_{v}*g_{v})^{s}=\alpha^2(f_{v}-g_{v})^{s}+(f_{v}^{s}+f_{v}*g_{v})^{s}=\alpha^2(f_{v}-g_{v})^{s}+f_{v}^{s}(f_{v}-g_{v})^{s}=(f_{v}-g_{v})^{s}(\alpha^2+f_v^2).
$$
Since $(f_{v}-g_{v})^{s}\not\equiv 0$, there exists $\alpha\in\Sm_{\R}(\Omega)$ such that $\alpha(f_{v}-g_{v})+f_{v}^{s}+f_{v}*g_{v}$ is a zero divisor if and only if $\alpha^{2}+f_{v}^{s}\equiv 0$ and, using the function $\mathcal{J}$,
 last equality is equivalent to saying that $f_{v}^{s}$
has a square root.
\end{proof}
For a detailed study of the existence of a square root for slice preserving functions see~\cite[Section 3]{A-dF}.

We now describe the image of $\mathcal{S}_{f,g}$, giving  necessary and sufficient conditions 
on $\mathfrak{b}$ for the existence of a solution of the equation $\mathcal{S}_{f,g}(\chi)=\mathfrak{b}$ together with
an explicit description of a particular solution.

\begin{proposition}
Let $f,g\in\Sm(\Omega)$ with $f_{0}= -g_{0}$ and $f_{v}^{s}=g_{v}^{s}$. Then $\mathcal{S}_{f,g}(\chi)=\mathfrak{b}$ has a solution if and only if  
$$f^{c}*\mathfrak{b}+\mathfrak{b}*g\equiv 0.$$
\end{proposition}

\begin{proof}
If $\chi$ is a solution of $\mathcal{S}_{f,g}(\chi)=\mathfrak{b}$, then $\mathfrak{b}=f*\chi+\chi*g$. We now have
\begin{align*}
f^{c}*\mathfrak{b}+\mathfrak{b}*g&= f^{c}*(f*\chi+\chi*g)+(f*\chi+\chi*g)*g\\
&=f^{s}\chi+f^{c}*\chi*g+f*\chi*g+\chi*g^{*2}\\
&=f^{s}\chi+2f_{0}\chi*g+\chi*g^{*2}\\
&=\chi*(f_{0}^{2}+f_{v}^s+2f_{0}g_{0}+2f_{0}g_{v}+g_{0}^{2}-g_{v}^{s}+2g_{0}g_{v})\equiv0,
\end{align*}
since $f_{0}=-g_{0}$ and $f_{v}^{s}=g_{v}^{s}$.

Assume now that $f^{c}*\mathfrak{b}+\mathfrak{b}*g\equiv 0$. We prove that $\mathfrak{b}$ belongs to the image of $\mathcal{S}_{f,g}$
by giving a different description of this linear subspace via the matrix $S_{f,g}$.
Thanks to our hypotheses and to Proposition~\ref{ranksimeq}, we have that  $S_{f,g}$ is 
skew symmetric and has rank 2. We now look for a square matrix $M$ whose kernel coincides with the image 
of $S_{f,g}$, which means $\mbox{rk} M=2$ and $M\cdot S_{f,g}=0$.  Then $\mathfrak{b}$ belongs to the image of $S_{f,g}$ if and only if
it belongs to $\ker M$. Since $f_{v}^{s}=g_{v}^{s}$, a straightforward computation shows that 
$$M=\begin{vmatrix}
f_{3} -g_{3}& -(f_{2}+g_{2})&f_{1}+g_{1}&0\\
f_{1}-g_{1} &0&-(f_{3}+g_{3})&f_{2}+g_{2}\\
f_{2}-g_{2}&f_{3}+g_{3}&0&-(f_{1}+g_{1})\\
0&f_{1}-g_{1}&f_{2}-g_{2}&f_{3}-g_{3}
\end{vmatrix},
$$
satisfies $M\cdot S_{f,g}=0$. In particular the image of $S_{f,g}$ is contained in the kernel of $M$ which therefore has 
rank less or equal than 2. Since at least one between $f_{v}$ and $g_{v}$ is not identically zero, then, by direct inspection
we have that $\text{rk}M=2$ which ensures that the image of $S_{f,g}$ coincides with $\ker M$.

Then writing $\mathfrak{b}=\mathfrak{b}_{0}+\mathfrak{b}_{1}i+\mathfrak{b}_{2}j+\mathfrak{b}_{3}k$ we obtain that $\mathcal{S}_{f,g}(\chi)=\mathfrak{b}$ has a solution
if and only if $F_{\mathcal{B}}(\mathfrak{b})\in\ker M$, that is
\begin{equation}\label{sysM}
\left\{\begin{tabular}{rrrrl}
$(f_{3} -g_{3})\mathfrak{b}_{0}$ & $-(f_{2}+g_{2})\mathfrak{b}_{1}$&$+(f_{1}+g_{1})\mathfrak{b}_{2}$& &$=0$\\
$(f_{1}-g_{1})\mathfrak{b}_{0}$ & &$-(f_{3}+g_{3})\mathfrak{b}_{2}$&$+(f_{2}+g_{2})\mathfrak{b}_{3}$&$=0$\\
$(f_{2}-g_{2})\mathfrak{b}_{0} $&$+(f_{3}+g_{3})\mathfrak{b}_{1}$& &$-(f_{1}+g_{1})\mathfrak{b}_{3}$&$=0$\\
&$(f_{1}-g_{1})\mathfrak{b}_{1}$&$+(f_{2}-g_{2})\mathfrak{b}_{2}$&$+(f_{3}-g_{3})\mathfrak{b}_{3}$&$=0$
\end{tabular}
\right.
\end{equation}

We now claim that the above system is a translation in coordinates of the equality $f^{c}*\mathfrak{b}+\mathfrak{b}*g\equiv 0$.
First of all notice, since 
$f_{0}=-g_{0}$, the equality $f^{c}*\mathfrak{b}+\mathfrak{b}*g\equiv 0$ can also be written as $f_{v}*\mathfrak{b}-\mathfrak{b}*g_{v}\equiv 0$.
By writing $\mathfrak{b}=\mathfrak{b}_{0}+\mathfrak{b}_{v}$ and splitting the ``real'' and ``vector'' parts of $f_{v}*\mathfrak{b}-\mathfrak{b}*g_{v}\equiv 0$, we obtain the equivalent system
$$
\begin{cases}
\langle f_{v},\mathfrak{b}_{v}\rangle_{*}-\langle g_{v},\mathfrak{b}_{v}\rangle_{*}\equiv0\\
\mathfrak{b}_{0}f_{v}+f_{v}\pv\mathfrak{b}_{v}-\mathfrak{b}_{0}g_{v}-\mathfrak{b}_{v}\pv g_{v}\equiv0.
\end{cases}
$$
The properties of the scalar product $\langle.,.\rangle_{*}$ and of the $\pv$-product yield
$$
\begin{cases}
\langle f_{v}-g_{v},\mathfrak{b}_{v}\rangle_{*}\equiv0\\
\mathfrak{b}_{0}(f_{v}-g_{v})+(f_{v}+g_{v})\pv\mathfrak{b}_{v}\equiv0.
\end{cases}
$$
A direct check shows that, up to a rearrangements of lines, this last system coincides with system~\eqref{sysM} 
\end{proof}

Next proposition describes a family of particular solutions of the equation $\mathcal{S}_{f,g}(\chi)=\mathfrak{b}$.

\begin{proposition}\label{partsol}
Let $f,g\in\Sm(\Omega)$ with $f\simeq-g$. If $f^{c}*\mathfrak{b}+\mathfrak{b}*g\equiv 0$,
then for any $h=h_{v},k=k_{v}\in\Sm(\Omega)$, such that 
$\langle f_{v},h_{v}\rangle_{*}+\langle g_{v},k_{v}\rangle_{*}\not\equiv 0$,
we have that
$$\chi=-(2\langle f_{v},h_{v}\rangle_{*}+2\langle g_{v},k_{v}\rangle_{*})^{-1}(h*\mathfrak{b}+\mathfrak{b}*k)$$ 
is a solution of $\mathcal{S}_{f,g}(\chi)=\mathfrak{b}$.
\end{proposition}

\begin{proof}
Being $\langle f_{v},h_{v}\rangle_{*}+\langle g_{v},k_{v}\rangle_{*}\in\Sm_{\R}(\Omega)\setminus\{0\}$, then
$-(2\langle f_{v},h_{v}\rangle_{*}+2\langle g_{v},k_{v}\rangle_{*})^{-1}(h*\mathfrak{b}+\mathfrak{b}*k)$ is well defined.
As $f_{0}=-g_{0}$ and $f_{v}*\mathfrak{b}=\mathfrak{b}*g_{v}$, the thesis is an immediate consequence of the following chain of
equalities
\begin{align*}
\mathcal{S}_{f,g}(h*\mathfrak{b}+\mathfrak{b}*k)&=f*(h*\mathfrak{b}+\mathfrak{b}*k)+(h*\mathfrak{b}+\mathfrak{b}*k)*g\\
&=f_{0}(h*\mathfrak{b}+\mathfrak{b}*k)+g_{0}(h*\mathfrak{b}+\mathfrak{b}*k)+f_{v}*(h*\mathfrak{b}+\mathfrak{b}*k)+(h*\mathfrak{b}+\mathfrak{b}*k)*g_{v}\\
&=f_{v}*h*\mathfrak{b}+f_{v}*\mathfrak{b}*k+h*\mathfrak{b}*g_{v}+\mathfrak{b}*k*g_{v}\\
&=f_{v}*h*\mathfrak{b}+\mathfrak{b}*g_{v}*k+h*f_{v}*\mathfrak{b}+\mathfrak{b}*k*g_{v}\\
&=(f_{v}*h_{v}+h_{v}*f_{v})*\mathfrak{b}+\mathfrak{b}*(g_{v}*k_{v}+k_{v}*g_{v})\\
&=-2\langle f_{v},h_{v}\rangle_{*}*\mathfrak{b}-\mathfrak{b}*2\langle g_{v},k_{v}\rangle_{*}=-2(\langle f_{v},h_{v}\rangle_{*}+\langle g_{v},k_{v}\rangle_{*})\mathfrak{b}.
\end{align*}
\end{proof}

\begin{remark}
Notice that there always exist $h,k\in\Sm(\Omega)$, with $h_{0}=k_{0}=0$,
such that the condition $\langle f_{v},h_{v}\rangle_{*}+\langle g_{v},k_{v}\rangle_{*}\not\equiv 0$ is satisfied.
Indeed, since $f_{v}\not\equiv0$, it is enough to take $k_{v}\equiv 0$ and $h=h_{v}\equiv\delta\in\SF$ such that
$(f_{v}\delta)_{0}=-\langle f_{v},\delta\rangle_{*} \not\equiv 0$. 
\end{remark}

The following corollary describes two special cases.

\begin{corollary}
Let $f,g\in\Sm(\Omega)$ be such that $f\simeq-g$ and assume $f^{c}*\mathfrak{b}+\mathfrak{b}*g\equiv 0$.
\begin{enumerate}
\item If $f_{v}$ is not a zero divisor, then $\chi=-(2f_{v}^{s})^{-1}(f_{v}*\mathfrak{b})$ is a solution of $\mathcal{S}_{f,g}(\chi)=\mathfrak{b}$.
\item For any $\delta\in\SF$ such that $(f\delta)_{0}\not\equiv 0$, then $\chi=-(2f\delta)_{0}^{-1}(\delta*\mathfrak{b})$ is a solution of $\mathcal{S}_{f,g}(\chi)=\mathfrak{b}$.
\end{enumerate}

\end{corollary}
\begin{proof}
In case \textit{(1)} take $h=f_{v}$ and $k\equiv 0$ in the statement of Proposition~\ref{partsol}; in case \textit{(2)} take
$h\equiv\delta$ and $k\equiv0$.\end{proof}

\section{Applications of the rank 2 case to function theory}\label{outcome}

The following result, which allows us to classify all idempotents up to $*$-conjugation, is a first application of the characterization of the equivalence relation $\simeq$ in terms of 
``real" and ``vector" parts of the functions, namely Corollary~\ref{sufficiencyproduct}.
\begin{proposition}\label{idemconjug}
Let $f\in\Sm(\Omega)\setminus\Sm_\R(\Omega)$; then $f$ is equivalent to a one-slice preserving function $g\in\Sm(\Omega)\setminus\Sm_\R(\Omega)$ if and only if $f_v^s\not\equiv 0$ has a square root. Moreover, all idempotents in $\So(\Omega)$ are equivalent. 
\end{proposition}
\begin{proof}
By Corollary~\ref{sufficiencyproduct}, the function $f$ is equivalent to $g$ if and only iff $f_0=g_0$ and $f_v^s=g_v^s$. Then it is enough to notice that for a one-slice preserving function $g\notin\Sm_\R(\Omega)$ we have $g_v=\gamma I$ for a suitable $I\in\SF$ and $\gamma\in\Sm_\R(\Omega)\setminus\{0\}$.

As for the second part of the statement, given an idempotent $\sigma$ and any $I\in\SF$, we have $\sigma_0=\ell^{+,I}_0=\frac12$ and $\sigma_v^s=(\ell^{+,I}_v)^s=-\frac14$, so that $\sigma\simeq \ell^{+,I}$.
\end{proof}

The previous proposition gives us the possibility to give a necessary and sufficient condition in order that the product of an idempotent with a function is identically zero. It is worth comparing this result with the statement of Proposition~\ref{kernelfg} in which the kernel
of $\mathcal{L}_{f,g}$ is characterized via a condition, while next theorem gives an extensional description. 

\begin{theorem}\label{ker-dimension}
Given an idempotent $\sigma\in\So(\Omega)$ and  $\rho\in\Sm(\Omega)$, then 
\begin{enumerate}
\item $\sigma*\rho\equiv 0$ if and only if there exist $I,J\in\SF$ with $I\perp J$, $\alpha,\beta\in\Sm_\R(\Omega)$ and $f\in\Sm(\Omega)$ invertible such that $\sigma=f*\ell^{+,I}*f^{-*}$ and $\rho=f*\ell^{-,I}*(\alpha+\beta J)*f^{-*}$. In particular, $\rho$ is an idempotent if and only if $\alpha=1$.
\item $\sigma*\rho*\sigma^c\equiv 0$ if and only if there exist $I,J\in\SF$ with $I\perp J$, $\alpha_0,\alpha_1,\beta\in\Sm_\R(\Omega)$ and $f\in\Sm(\Omega)$ invertible such that $\sigma=f*\ell^{+,I}*f^{-*}$ and $\rho=f*(\alpha_0+\alpha_1I+\beta \ell^{-,I}*J)*f^{-*}$. In particular, $\rho$ is an idempotent if and only if $\alpha_0=\frac12$ and $\alpha_1^{2}=-\frac14$.
\item $\sigma*\rho*\sigma\equiv 0$ if and only if there exist $I,J\in\SF$ with $I\perp J$, $\alpha_0,\beta_2,\beta_3\in\Sm_\R(\Omega)$ and $f\in\Sm(\Omega)$ invertible such that $\sigma=f*\ell^{+,I}*f^{-*}$ and $\rho=f*(\alpha\ell^{-,I}+(\beta_2+\beta_3i)*J)*f^{-*}$. In particular, $\rho$ is an idempotent if and only if $\alpha=1$ and $\beta_{2}^{2}+\beta_{3}^{2}\equiv0$.
\end{enumerate}
\end{theorem}

\begin{proof}
\textit{(1)}. A direct computation shows that, if $\sigma=f*\ell^{+,I}*f^{-*}$ and $\rho=f*\ell^{-,I}*(\alpha+\beta J)*f^{-*}$, then 
$\sigma*\rho=f*\ell^{+,I}*\ell^{-,J}*(\alpha+\beta J)*f^{-*}\equiv 0$ because $\ell^{+,I}*\ell^{-,I}\equiv 0$. 

Vice versa, Proposition~\ref{idemconjug} entails that if $\sigma$ is a idempotent, there exist $f\in\Sm(\Omega)$ invertible such that $\sigma=f*\ell^{+,I}*f^{-*}$. As $\sigma*\rho\equiv0$ iff $f^{-*}*\sigma*\rho*f\equiv0$, we can reduce ourselves to the case $f=1$, that is $\sigma=\ell^{+,I}$.
Now set $\rho=\rho_0+\rho_1I+\rho_2J+\rho_3K$ and compute 
\begin{align*}
\ell^{+,I}*\rho&=
\frac12\left(1-\mathcal{J}i\right)*\left(\rho_0+\rho_v\right)
=\frac12\left[\rho_0+\langle\mathcal{J}I,\rho_v\rangle_*+\rho_v-\mathcal{J}\rho_0I-\mathcal{J}I\pv \rho_v\right]\\
&=\frac12\left[\rho_0+\mathcal{J}\rho_1+(\rho_1I+\rho_2J+\rho_3K)-\mathcal{J}\rho_0I-\mathcal{J}(-\rho_3J+\rho_2K)\right]\\
&=\frac12\left[\rho_0+\mathcal{J}\rho_1+(\rho_1-\mathcal{J}\rho_0)I+(\rho_2+\mathcal{J}\rho_3)J+(\rho_3-\mathcal{J}\rho_2)K\right].
\end{align*}
Hence we obtain that $\ell^{+,I}*\rho\equiv0$ if and only if 
$$\begin{cases}
\rho_0+\mathcal{J}\rho_1\equiv 0,\\
\rho_1-\mathcal{J}\rho_0\equiv 0,\\
\rho_2+\mathcal{J}\rho_3\equiv 0\\
\rho_3-\mathcal{J}\rho_2\equiv 0.
\end{cases}
$$
This system is equivalent to $\rho_1=\mathcal{J}\rho_0$ and $\rho_3=\mathcal{J}\rho_2$ and these last two equalities give 
$$
\rho=\rho_0+\mathcal{J}\rho_0I+\rho_2J+\rho_2\mathcal{J}K=\rho_0(1+\mathcal{J}I)+\rho_2(1+\mathcal{J}I)J;
$$ 
by setting $\alpha=2\rho_0$ and $\beta=2\rho_2$ we get $\rho=\ell^{-,I}*(\alpha+\beta J)$.
Finally, $\rho=f*\ell^{-,I}*(\alpha+\beta J)*f^{-*}$ is an idempotent if and only if $\ell^{-,I}*(\alpha+\beta J)$ is, and a straightforward computation shows that this holds if and only if $\alpha=1$.

\textit{(2)}. Again a direct computation shows that the condition is sufficient.

Vice versa, as above we can suppose that $\sigma=\ell^{+,I}$; writing $\rho=\rho_0+\rho_1I+\rho_2J+\rho_3K$ we obtain, since $\ell^{+,I}$ is an idempotent and $\ell^{+,I}*\ell^{-,I}\equiv0$,
\begin{align*}
\ell^{+,I}*\rho*\ell^{-,I}&=
\ell^{+,I}*(\rho_0+\rho_1I+\rho_2J+\rho_3K)*\ell^{-,I}
=\ell^{+,I}*(\rho_0+\rho_1I)*\ell^{-,I}+\ell^{+,I}*(\rho_2J+\rho_3K)*\ell^{-,I}\\
&=(\rho_0+\rho_1I)*\ell^{+,I}*\ell^{-,I}+\rho_2\ell^{+,I}*J*\ell^{-,I}+\rho_3\ell^{+,I}*K*\ell^{-,I}\\
&=\rho_2\ell^{+,I}*\ell^{+,I}*J+\rho_3\ell^{+,I}*\ell^{+,I}*K
=\rho_2\ell^{+,I}*J+\rho_3*\ell^{+,I}*K=\ell^{+,I}*(\rho_2+\rho_3I)*J.
\end{align*}
Thus $\ell^{+,I}*\rho*\ell^{-,I}\equiv0$ if and only if $\ell^{+,I}*(\rho_2+\rho_3I)\equiv0$ which, thanks to \textit{(1)}, gives the existence of a suitable $\beta\in\Sm_\R(\Omega)$ such that  $(\rho_2+\rho_3I)*J=\beta\ell^{-,I}* J$ and thus proves the first part of the assertion.
Again $\rho=f*(\alpha_{0}+\alpha_{1}I+\beta\ell^{-,I}*J)*f^{-*}$ is an idempotent if and only if $\alpha_{0}+\alpha_{1}I+\beta\ell^{-,I}*J$ is
and this is equivalent to $\alpha_{0}=\frac12$ and $\alpha_{1}^{2}=-\frac14$.

\textit{(3)}. The sufficiency of the condition is proved by direct inspection, as above. 

We only give a short summary of the computations, since the procedure is the same as in case \textit{(2)}

\begin{align*}
\ell^{+,I}*\rho*\ell^{+,I}&=
\ell^{+,I}*(\rho_0+\rho_1I+\rho_2J+\rho_3K)*\ell^{+,I}
=\ell^{+I}*(\rho_0+\rho_1I)*\ell^{+,I}+\ell^{+,I}*(\rho_2J+\rho_3K)*\ell^{+,I}\\
&=(\rho_0+\rho_1I)*\ell^{+,I}*\ell^{+,I}+\rho_2\ell^{+,I}*J*\ell^{+,I}+\rho_3\ell^{+,I}*K*\ell^{+,I}\\
&=(\rho_0+\rho_1I)*\ell^{+,I}+\rho_2\ell^{+,I}*\ell^{-,I}*J+\rho_3\ell^{+,I}*\ell^{-,I}*K
=(\rho_0+\rho_1I)*\ell^{+,I}.
\end{align*}
Thus $\ell^{+,I}*\rho*\ell^{+,I}\equiv0$ if and only if $(\rho_0+\rho_1I)*\ell^{+,I}\equiv0$ which is equivalent to $\rho_0+\rho_1I=\alpha \ell^{-I}$ for a suitable $\alpha\in\Sm_\R(\Omega)$.

\end{proof}

\begin{remark}
The above proposition classifies, up to conjugation, all functions $\sigma,\rho$ such that $\sigma$ is an idempotent and 
$\sigma*\rho\equiv0$ showing that, up to conjugation, $\sigma=\ell^{+,I}$ and $\rho=\ell^{-,I}*(\alpha+\beta J)$ with $I,J\in\SF$, $I\perp J$, $\alpha,\beta\in\Sm_R(\Omega)$. Notice that for these functions $\rho*\sigma$ can be different from $0$. Indeed, $\rho*\sigma\equiv0$ iff $\ell^{-,I}*(\alpha+\beta J)*\ell^{+,I}=\alpha\ell^{-,I}*\ell^{+,I}+\beta\ell^{-,I}*J*\ell^{+,I}\equiv0$. Since $\ell^{-,I}*\ell^{+,I}\equiv0$ we have $\rho*\sigma\equiv0$ if and only if $\beta\ell^{-,I}*J*\ell^{+,I}\equiv0$. As $J$ is orthogonal to $\ell^{-,I}$ we get $J*\ell^{+,I}=\ell^{-,I}*J$ and thus $\beta\ell^{-,I}*J*\ell^{+,I}\equiv0$ is equivalent to $\beta\ell^{-,I}*\ell^{-,I}*J=\beta\ell^{-,I}*J\equiv0$, since $\ell^{-,I}$ is an idempotent. Thus $\rho*\sigma\equiv 0$ iff $\beta\equiv0$, which is equivalent to $\rho=\alpha\ell^{-,I}$. Again, $\rho$ is an idempotent if and only if $\alpha=1$, that is $\rho=\sigma^c$.
\end{remark}

\section{Sylvester operators of rank $3$}\label{rank3sect}

We are now left to investigate more precisely the case when the Sylvester operator $\mathcal{S}_{f,g}$ has rank $3$. Thanks to Theorem~\ref{rank3}, this corresponds to the fact that $f_0+g_0\not\equiv0$ and  $\mathcal{S}_{f,g}$ is not an isomorphism. We recall that by Remark~\ref{noproduct} and Proposition~\ref{ranksimeq}, this can happen only if $\Omega$ is a product domain.
Since $f_0+g_0\in\Sm_\R(\Omega)\setminus\{0\}$ is invertible, with no loss of generality  we can study the kernel and the image of the operator associated to the functions $\frac{f}{f_0+g_0}$ and $\frac{g}{f_0+g_0}$, that is we can assume $f_0+g_0\equiv1$.

Next result gives necessary conditions on the functions $f$ and $g$ in order that  $S_{f,g}$ is not an isomorphism.

\begin{proposition}\label{squareroot}
Assume that $f_0+g_0\equiv 1$ and $S_{f,g}$ is not an isomorphism. Then there exists $\tau\in\Sm_\R(\Omega)$ such that $f_v^s=\left(\mathcal{J}\left(\tau-\frac12\right)\right)^2$ and $g_v^s=\left(\mathcal{J}\left(\tau+\frac12\right)\right)^2$; in particular both $f_v^s$ and $g_v^s$ have a square root in $\in\Sm_\R(\Omega)$.
\end{proposition}
\begin{proof}
Under the assumption on $f_0+g_0$, the determinant of $S_{f,g}$ becomes $1+2(f_v^s+g_v^s)+(f_v^s-g_v^s)^2$
 which can also be written as $(f_v^s+g_v^s+1)^2-4f_v^sg_v^s$.

As $\mathcal{S}_{f,g}$ is not an isomorphism, then we have $(f_v^s+g_v^s+1)^2-4f_v^sg_v^s\equiv0$, which implies that $f_v^sg_v^s$ has a square root $\mu\in\Sm_\R(\Omega)$.

Up to a possible change of sign of $\mu$ we have that the following system holds
\begin{equation*}
\begin{cases}
f_v^s+g_v^s+1=2\mu,\\
f_v^sg_v^s=\mu^2
\end{cases}
\end{equation*}
The first equality gives $g_v^s=2\mu-1-f_v^s$, and thanks the second one, we obtain
$f_v^s(2\mu-1-f_v^s)=\mu^2$. Last equation is equivalent to 
$(f_v^s)^2-2(\mu-\frac12)f_v^s+\mu^2-\mu+\frac14\equiv-\mu+\frac14$
which can also be written as 
$\left(f_v^s-\mu+\frac12\right)^2=-\mu+\frac14$, thus showing that $-\mu+\frac14$ has a square root $\tau\in\Sm_\R(\Omega)$. Up to a change of sign of $\tau$, it holds  $f_v^s-\mu+\frac12=\tau$, that is $f_v^s=\mu-\frac12+\tau$. As $\mu-\frac12=-\tau^2-\frac14$ we finally obtain that 
$$
f_v^s=-\tau^2-\frac14+\tau=-\left(\tau-\frac12\right)^2=\left(\mathcal{J}\left(\tau-\frac12\right)\right)^2
$$ 
which therefore proves that $f_v^s$ has a square root. Since $g_v^s=2\mu-1-f_v^s$, we have 
$$
g_v^s=2\left(-\tau^2-\frac14\right)+\left(\tau-\frac12\right)^2=-\tau^2-\frac14-\tau=-\left(\tau+\frac12\right)^2=\left(\mathcal{J}\left(\tau+\frac12\right)\right)^2,
$$
 showing that $g_v^s$ has also the required form and admits a square root.
\end{proof}

Last proposition gives us the possibility to study more accurately which are the functions $f$ and $g$ such that $f_0+g_0\equiv1$ and  $\mathcal{S}_{f,g}$ is not invertible. The crucial point is that this analysis must be split in two parts, corresponding to~Examples~\ref{rank3ex1} and ~\ref{rank3ex2}: indeed the main difference we will find is that in the first case $f_v^sg_v^s\not\equiv0$, thus ensuring that the eigenvalue $0$ has algebraic multiplicity $1$, while in the second one $f_v^sg_v^s\equiv0$, which entails that the eigenvalue $0$ has algebraic multiplicity greater than $1$.

\begin{proposition}\label{algmult1}
Assume that $f_0+g_0\equiv 1$ and $S_{f,g}$ is not an isomorphism. If $f_v^sg_v^s\not\equiv0$, then there exist $h,\tilde h\in\Sm(\Omega)$ 
invertible, $\tau\in\Sm_\R(\Omega)\setminus\left\{\pm\frac12\right\}$, $I\in\SF$ such that $h^{-*}*f_v*h=\mathcal{J}\left(\tau-\frac12\right)I$ 
and $\tilde h*g_v*\tilde h^{-*}=\mathcal{J}\left(\tau+\frac12\right)I$. Moreover for any $J\in\SF$ such that $I\perp J$ and $K=IJ$ we have
$$
\ker(\mathcal{S}_{f,g})=\left\{\alpha h*(\mathcal{J}J+K)*\tilde h\,|\, \alpha\in \Sm_\R(\Omega)\right\}
$$
and 
$\mathcal{S}_{f,g}=\mathfrak{b}$ has a solution if and only if 
$ \langle h^{-*}*\mathfrak{b}*\tilde h^{-*},\mathcal{J}J-K\rangle_*\equiv0$.
\end{proposition}
\begin{proof}
Thanks to Propositions~\ref{squareroot} and~\ref{idemconjug} we can find $h,\tilde h\in\Sm(\Omega)$ invertible, $\tau\in\Sm_\R(\Omega)\setminus\left\{\pm\frac12\right\}$, $I\in\SF$ such that $h^{-*}*f_v*h=\mathcal{J}\left(\tau-\frac12\right)I$ and $\tilde h*g_v*\tilde h^{-*}=\mathcal{J}\left(\tau+\frac12\right)I$. Thus, by a straightforward computation, it is enough to study the Sylvester operator $\mathcal{S}_{f,g}$ when $f_v=\mathcal{J}\left(\tau+\frac12\right)I$ and $g_v=\mathcal{J}\left(\tau-\frac12\right)I$ and to recover kernel and image in the general case from the kernel and the image associated to these specific functions.
The matrix $S_{f,g}$ in Formula~\ref{matrixiLR} is given by
$$
\begin{vmatrix}
1& -2\mathcal{J}\tau &0&0\\
2\mathcal{J}\tau &1&0&0\\
0&0&1&-\mathcal{J}\\
0&0&\mathcal{J}&1
\end{vmatrix}.
$$
As $\tau^2\neq \frac14$ we easily obtain that 
$\ker(\mathcal{S}_{f,g})$ is spanned by $\mathcal{J}J+K$ and that the image of $\mathcal{S}_{f,g}$ is spanned by $1$, $I$ and $J+\mathcal{J}K$. Last assertion can also be rephrased by saying that 
$\mathfrak{b}$ belongs to the image of $\mathcal{S}_{f,g}$ if and only if 
$ \langle \mathfrak{b},\mathcal{J}J-K\rangle_*\equiv0$.
\end{proof}


We recall that, thanks to Remark~\ref{rank3symmetric}, $\mathcal{S}_{f,g}$ has rank $3$ if and only if $\mathcal{S}_{g,f}$ has rank $3$. By Theorem~\ref{kerthm}, this condition is equivalent to the fact that $\ker(\mathcal{S}_{f,g})$ contains only zero divisors 
(indeed the existence of a non-zero element in $\ker(\mathcal{S}_{f,g})$ rules out the case $\text{rk}(\mathcal{S}_{f,g})=4$ and the absence of
invertible elements in $\ker(\mathcal{S}_{f,g})$ prevents $\text{rk}(\mathcal{S}_{f,g})=2$). Under this hypothesis, notice that, if there exists a zero divisor in $\ker(\mathcal{S}_{f,g})$ whose ``real part'' is not identically zero, then $\ker(\mathcal{S}_{f,g})$ contains exactly one idempotent. Quite surprisingly, this property is not symmetric in $f$ and $g$: in particular, we can find $f,g\in\Sm(\Omega)$ such that $f_{0}+g_{0}=1$ and $\ker(\mathcal{S}_{f,g})$ contains 
an idempotent while $\ker(\mathcal{S}_{g,f})$ only contains zero divisor with ``real part'' equal to zero.

With the same notation as in the statement of Proposition~\ref{algmult1}, we have that
$$
\ker(\mathcal{S}_{g,f})=\left\{\alpha \tilde h^{-*}*(\mathcal{J}J-K)* h^{-*}\,|\, \alpha\in \Sm_\R(\Omega)\right\}=\left\{\alpha \tilde h^{c}*(\mathcal{J}J-K)* h^{c}\,|\, \alpha\in \Sm_\R(\Omega)\right\}.
$$
Let us compute the ``real part'' of the elements of $\ker(\mathcal{S}_{f,g})$ and  $\ker(\mathcal{S}_{g,f})$.
Factoring out the slice preserving function $\alpha$ we have,
\begin{align}
(h*(\mathcal{J}J+K)*\tilde h)_{0}&=((h_{0}+h_{v})*(\mathcal{J}J+K)*(\tilde h_{0}+\tilde h_{v}))_{0}\nonumber\\
&=\left(\left(-\langle h_{v},\mathcal{J}J+K\rangle_{*}+h_{0}(\mathcal{J}J+K)+h_{v}\pv(\mathcal{J}J+K)\right)*(\tilde h_{0}+\tilde h_{v})\right)_{0}\nonumber\\
&=-\tilde h_{0}\langle h_{v},\mathcal{J}J+K\rangle_{*}-\langle h_{0}(\mathcal{J}J+K)+h_{v}\pv(\mathcal{J}J+K),\tilde h_{v}\rangle_{*}\nonumber\\
&=-\tilde h_{0}\langle h_{v},\mathcal{J}J+K\rangle_{*}-h_{0}\langle \mathcal{J}J+K,\tilde h_{v}\rangle_{*}-\det\begin{vmatrix}
 h_{v} & (\mathcal{J}J+K) &\tilde h_{v}
\end{vmatrix}.\label{eqpartereale1}
\end{align}
Analogously we have 
\begin{equation}
\label{eqpartereale2}
(\tilde h^{c}*(\mathcal{J}J-K)* h^{c})_{0}=
\tilde h_{0}\langle h_{v},\mathcal{J}J-K\rangle_{*}+h_{0}\langle \mathcal{J}J-K,\tilde h_{v}\rangle_{*}-\det\begin{vmatrix}
 \tilde h_{v} & (\mathcal{J}J-K) & h_{v}\end{vmatrix}.
\end{equation}


\begin{example}
Take $h=(\mathcal{J}-1)+i+j$ and $\tilde{h}=i+k$. Then $h_{0}=(\mathcal{J}-1)$, $h_{v}=i+j$, $\tilde h_{0}=0$, $\tilde h_{v}=i+k$, $h^{s}=(\mathcal{J}-1)^{2}+1+1=2-2\mathcal{J}$ and $\tilde h^{s}=2$. Then equation~\eqref{eqpartereale1} gives
$$(h*(\mathcal{J}j+k)*\tilde h)_{0}=-(\mathcal{J}-1)\cdot 1-(\mathcal{J}+1)=-2\mathcal{J},$$
while equation~\eqref{eqpartereale2} gives
$$
(\tilde h^{c}*(\mathcal{J}j-k)* h^{c})_{0}=(\mathcal{J}-1)\cdot (-1)-(-\mathcal{J}+1)=0.
$$
Thus for any $\tau\in\Sm_{\R}(\Omega)$, given $f=1+h*(\mathcal{J}\left(\tau-\frac12\right)i)*h^{-*}$ and $g=\tilde h^{-*}*(\mathcal{J}\left(\tau+\frac12\right)i)*\tilde h$, we have that $\ker(\mathcal{S}_{g,f})$ contains only zero divisors with vanishing ``real part'',
while $\ker(\mathcal{S}_{f,g})$ contains an idempotent.
\end{example}

We are now left to deal with the condition $f_v^sg_v^s\equiv0$. We will examine thoroughly the case $g_v^s\equiv 0$, while the symmetrical one $f_v^s\equiv 0$ is left to the reader.

\begin{proposition}
Assume that $f_0+g_0\equiv 1$ and $S_{f,g}$ is not an isomorphism. If $g_v^s\equiv0$, then 
\begin{equation}\label{kerrank3-case2}
\ker(\mathcal{S}_{f,g})=\left\{(1-f_v)*X*g_v\,|\, X\in \Sm(\Omega)\right\} 
\end{equation}
and $\mathcal{S}_{f,g}=\mathfrak{b}$ has a solution if and only if $(1-f_v)*\mathfrak{b}*g_v\equiv0$.
\end{proposition}

\begin{proof} 
Thanks to Propositions~\ref{squareroot} and~\ref{idemconjug} we can find $h\in\Sm(\Omega)$ invertible and $I\in\SF$ such that $h^{-*}*f_v*h=-\mathcal{J}I$, and hence $1+f_v$ is a zero divisor. 
Moreover, since $\mathcal{S}_{f,g}(\chi)=(f_0+g_0)\chi+f_v*\chi+\chi*g_v=1\cdot\chi+f_v*\chi+\chi*g_v=(1+f_v)*\chi+\chi*g_v$, a trivial computation shows that for any $X\in \Sm(\Omega)$ the following chain of equality holds
\begin{align*}
\mathcal{S}_{f,g}((1-f_v)*X*g_v)&=(1+f_v)*(1-f_v)*X*g_v+(1-f_v)*X*g_v*g_v\\
&=((1+f_v)*(1-f_v))*X*g_v+(1-f_v)*X*(g_v*g_v)\\
&=(1+f_v)^s*X*g_v+(1-f_v)*X*(-g_v^s)=0,
\end{align*} 
and therefore $(1-f_v)*X*g_v\in\ker(\mathcal{S}_{f,g})$ for any $X\in \Sm(\Omega)$. 

We now claim that there exist $X\in \Sm(\Omega)$ such that
 $(1-f_v)*X*g_v$ is not identically zero. Indeed, since $g_v\not\equiv0$ we can find $I\in\SF$ such that $g_v*I$ has non-zero real part, so there exists $\tilde h\in\Sm(\Omega)$ invertible that $\tilde h^{-*}*g_v*I*\tilde h$ is a non-zero ``real'' multiple of $1-\mathcal{J}I$. Moreover we already know that there exists $h\in\Sm(\Omega)$ invertible such that $h^{-*}*(1+f_v)*h=1-\mathcal{J}I$.
Thus $(1-f_v)*X*g_v\not\equiv0$ if and only if $h^{-*}*(1-f_v)*X*(g_v*I)*\tilde h\not\equiv0$, so that last inequality is equivalent to 
$(1-\mathcal{J}I)*h^{-*}X*\tilde h*(1-\mathcal{J}I)\not\equiv0$. Now, up to a factor $4$, we have $\sigma*h^{-*}X*\tilde h*\sigma\not\equiv0$ for the idempotent $\sigma=\frac12(1-\mathcal{J}I)$ and taking $X=h*\sigma*\tilde h^{-*}$
gives $\sigma*\sigma*\sigma=\sigma\not\equiv0$.

As $\ker(\mathcal{S}_{f,g})$ has dimension $1$ and $(1-f_v)*X*g_v$ is different from zero for some $X\in\Sm(\Omega)$, the equality in Formula~\ref{kerrank3-case2} is established.

We are now left to consider the image of the operator $\mathcal{S}_{f,g}$. First of all notice that, given $\chi\in\Sm(\Omega)$ we have that 
\begin{align*}
(1-f_v)*\mathcal{S}_{f,g}(\chi)*g_v
&=(1-f_v)*((1+f_v)*\chi+\chi*g_v)*g_v\\
&=(1-f_v)*(1+f_v)*\chi*g_v+(1-f_v)*\chi*g_v*g_v\\
&=(1+f_v)^s\chi*g_v+(1-f_v)*\chi*(g_v^s)=0,
\end{align*}
because both $1+f_v$ and $g_v$ are zero divisors. 
Thus if $\mathcal{S}_{f,g}(\chi)=\mathfrak{b}$ has a solution then  $(1-f_v)*\mathfrak{b}*g_v\equiv0$, showing that the image of $\mathcal{S}_{f,g}$ is contained in the linear subspace $\left\{\mathfrak{b}\in\Sm(\Omega)\,|\, (1-f_v)*\mathfrak{b}*g_v\equiv0\right\}$.

Reasoning as before, Theorem~\ref{ker-dimension} ensures that the dimension of $\left\{\mathfrak{b}\in\Sm(\Omega)\,|\, (1-f_v)*\mathfrak{b}*g_v\equiv0\right\}$ is equal to $3$, and hence the image of $S_{f,g}$ coincides with $\left\{\mathfrak{b}\in\Sm(\Omega)\,|\, (1-f_v)*\mathfrak{b}*g_v\equiv0\right\}$, thus completing the proof of the statement. 
\end{proof}


\bibliographystyle{amsplain}

\begin{thebibliography}{Ci-Kr}


\bibitem{ACSbook} D.~Alpay, F.~Colombo, I.~Sabadini. Slice Hyperholomorphic Schur Analysis, Oper. Theory Adv.
Appl., vol. 256, Birkh\"auser Basel, 2016.

\bibitem{A-CVEE}
A. Altavilla, Some properties for quaternionic slice-regular functions on domains without real points. 
Complex Var. Elliptic Equ. 60, n. 1 (2015), 59--77.

\bibitem{altavilladiff}
A. Altavilla.
On the real differential of a slice regular function.
Adv. Geom. 18 (2018), no. 1, 5--26. 


\bibitem{A-dF} A. Altavilla, C. de Fabritiis, $*$-exponential of slice-regular functions, Proc. Amer. Math. Soc. 147, 2019, 1173--1188.
\bibitem{A-dFAMPA} A. Altavilla, C. de Fabritiis, s-Regular functions which preserve a complex slice, Ann. Mat. Pura Appl. (4) 197:4, 2018, 1269--1294.

\bibitem{A-S}A. Altavilla and G. Sarfatti. Slice-Polynomial Functions and Twistor Geometry of Ruled Surfaces in $\mathbb{CP}^{3}$. Math. Z. 291(3-4) (2019), 1059--1092.

\bibitem{howandwhy} R. Bhatia, P. Rosenthal, How and why to solve the operator equation $AX - XB = Y$?  Bull. London Math. Soc. 29 (1997) 1--21.

\bibitem{bolotnikov1} V. Bolotnikov, Polynomial interpolation over quaternions, Journal of Mathematical Analysis and Applications,
421(1), 2015, 567--590.

\bibitem{bolotnikov2} V. Bolotnikov, On the Sylvester Equation over Quaternions,
Operator Theory: Advances and Applications, Volume 252, (2016), Pages 43--75.

\bibitem{C-G-S-St}  F. Colombo, G. Gentili, I. Sabadini, D. C. Struppa, Extension results for slice regular functions
of a quaternionic variable. Adv. Math. 222(5), (2009), 1793--1808.
\bibitem{C-GC-S} F. Colombo, J. Oscar Gonzalez-Cervantes, I. Sabadini, The C-property for slice regular functions and applications to the Bergman space, Compl. Var. Ell. Eq., 58, n. 10 (2013), 1355--1372.
\bibitem{C-S-St-1}  F. Colombo, I. Sabadini, D. C. Struppa, Noncommutative Functional Calculus, Progress In Mathematics, Birkhauser, 2011.
\bibitem{C-S-St-2}  F. Colombo, I. Sabadini, D. C. Struppa, Entire Slice Regular Functions, SpringerBriefs in Mathematics, Springer, 2016.
\bibitem{deF-G-S}
C. de Fabritiis, G. Gentili, G. Sarfatti, Quaternionic Hardy Spaces, Ann. SNS Pisa, 18 (2), (2018), pp. 679--733.

\bibitem{G-S} G. Gentili, C. Stoppato, Zeros of regular functions and polynomials of a quaternionic variable.
Mich. Math. J. 56(3) (2008), 655--667.
\bibitem{G-S-St} G. Gentili, C. Stoppato, D. C. Struppa, Regular Functions of a Quaternionic Variable,  Springer Monographs in Matehmatics, Springer, 2013. 
\bibitem{G-M-P} R. Ghiloni, V. Moretti, A. Perotti, Continuous Slice Functional Calculus in Quaternionic Hilbert Spaces,
Rev. Math. Phys. 25 (2013), 1350006-1-1350006-83.
\bibitem{G-P} R. Ghiloni, A. Perotti, Slice regular functions on real alternative algebras, Adv. in Math., v. 226, n. 2 (2011),  1662-1691.

\bibitem{G-P-wing} R. Ghiloni, A. Perotti, On a class of orientation-preserving maps of $\mathbb{R}^{4}$, J. Geom. Anal. (2020). https://doi.org/10.1007/s12220-020-00356-8.


\bibitem{GPSalgebra} R. Ghiloni, A. Perotti, C. Stoppato, The algebra of slice functions, Trans. of Amer. Math. Soc., Volume 369, N.7, 2017, pp.4725--4762.
\bibitem{GPSadvances} R. Ghiloni, A. Perotti, and C. Stoppato. Singularities of slice regular functions over real
alternative $*$-algebras. Adv. Math., 305:1085--1130, 2017.

\bibitem{GPSdivision} 
R. Ghiloni., A. Perotti, C. Stoppato, Division algebras of slice functions. Proceedings of the Royal Society of Edinburgh: Section A Mathematics, 150(4), (2020), 2055--2082. doi:10.1017/prm.2019.13.



\bibitem{he} Z.-H. He, J. Liu, T.-Y. Tam, The general $\phi$-hermitian solution to mixed pairs of quaternion matrix Sylvester equations,
Electronic Journal of Linear AlgebraOpen Access, Volume 32, 1 January 2017, Article number 36, Pages 475--499.






\bibitem{janovska}
D.~Janovsk\'a and G.~Opfer.
Linear equations in quaternionic variables. 
Mitt. Math. Ges. Hamburg 27 (2008), 223--234. 

\bibitem{M-S}
A.~Monguzzi, G.~Sarfatti, Shift invariant subspaces of slice $L^{2}$ functions. Ann. Acad. Sci. Fenn. Math. 43 (2018), 1045--1061.

\bibitem{stoppatopoles} C. Stoppato, Poles of regular quaternionic functions. Complex Var. Elliptic Equators. 54(11), 1001--1018, 2009.

\bibitem{stoppatosing} C. Stoppato, Singularities of slice regular functions, Math. Nachr., 285(10):1274--1293, 2012.

\bibitem {sylvester} 
J. Sylvester, Sur l'equations en matrices $px = xq$, C.R. Acad. Sci. Paris 99 (1884) 67--71, 115--116.

\end{thebibliography}

\end{document}